\documentclass{amsart}
\usepackage[utf8]{inputenc}

\usepackage{amsmath, amssymb, amsthm, mathtools, enumerate, enumitem, xcolor, import, hyperref, nicefrac}

\newcommand{\supp}{\operatorname{supp}}
\newcommand{\pardist}{\omega}

\newtheorem{theorem}{Theorem}[section]
\newtheorem{corollary}[theorem]{Corollary}
\newtheorem{prop}[theorem]{Proposition}
\newtheorem{lemma}[theorem]{Lemma}

\theoremstyle{definition}
\newtheorem{definition}[theorem]{Definition}

\newtheorem{notation}[theorem]{Notation}

\theoremstyle{remark}
\newtheorem{remark}[theorem]{Remark}

\makeatletter
\let\c@equation\c@theorem
\makeatother
\numberwithin{equation}{section}

\newcommand{\thistheoremname}{}
\newtheorem*{genericthm*}{\thistheoremname}
\newenvironment{namedthm*}[1]
{\renewcommand{\thistheoremname}{#1}%
	\begin{genericthm*}}
	{\end{genericthm*}}

\usepackage[textsize=tiny]{todonotes}

\title{Characterising 1-Rectifiable metric spaces via Connected Tangent Spaces}
\author{David Bate}
\email{david.bate@warwick.ac.uk}
\author{Phoebe Valentine}
\email{phoebe.valentine@warwick.ac.uk}
\address{Zeeman Building, University of Warwick, Coventry CV4 7AL}
\thanks{
David Bate is supported by the European Union's Horizon 2020 research and innovation programme (Grant agreement No. 948021).	
Phoebe Valentine is supported by the Warwick Mathematics Institute Centre for Doctoral Training, and gratefully acknowledges funding from the University of Warwick and the UK Engineering and Physical Sciences Research Council (Grant number: EP/W524645/1)}

\begin{document}	
	\begin{abstract}
		We prove that in a complete metric space $X$, $1$-rectifiability of a set $E\subset X$ with $\mathcal{H}^1(E)<\infty$ and positive lower density $\mathcal{H}^1$-a.e.\ is implied by the property that all tangent spaces are connected metric spaces.
	\end{abstract}
	\maketitle
	\setcounter{tocdepth}{1} %depth 1 means just sections, depth 2 means subsections etc
	
	\section{Introduction}
	A central concept in geometric measure theory is that of rectifiability. For a metric space $(X,d)$, a set $E\subset X$ is \emph{$n$-rectifiable} if there exists countably many Lipschitz maps $f_i \colon A_i\subset \mathbb{R}^n\rightarrow X$ such that
\[\mathcal{H}^n(E\backslash \bigcup_{i\in \mathbb{N}}f_i(A_i))=0.\]
Here and throughout, $\mathcal{H}^n$ denotes the $n$-dimensional Hausdorff (outer) measure. A set is \emph{purely $n$-unrectifiable} if its intersection with any $n$-rectifiable set has $\mathcal{H}^n$ measure zero.

This paper concerns characterizations of rectifiable metric spaces in terms of tangent spaces.
In classical geometric measure theory, that is, for subsets of Euclidean space, Preiss \cite{paper:preiss} established a theory of tangent \emph{measures}.
Further, the rectifiability criterion of Marstrand \cite{paper:marstrand} and Mattila \cite{paper:mattila} shows that a set $E\subset \mathbb{R}^m$ (with $\mathcal H^n(E)<\infty$ and positive lower density almost everywhere) is $n$-rectifiable if, for $\mathcal{H}^n$-a.e. $x\in E$, all tangent measures of $\mathcal{H}^n|_E$ at $x$ are supported on $n$-dimensional subspace of $\mathbb{R}^m$.

In the setting of metric spaces, the work of Kirchheim \cite{paper:kirchheim} shows that rectifiable sets possess a strong tangent structure.
In particular, this work implies that tangent spaces (defined as a suitable Gromov--Hausdorff type adaptation of tangent measures, see Section \ref{prelims}) of a $n$-rectifiable set are Lebesgue measure supported on a normed $\mathbb R^n$.
It remained an open question for sometime whether a result analogous to the Marstrand--Mattila theorem holds in metric spaces.
Recently, this was shown to be the case.
Indeed, Bate \cite{bate} shows that a metric space $E$ (with $\mathcal H^n$) is $n$-rectifiable whenever, for $\mathcal{H}^n$-a.e. $x\in E$, all tangent spaces of $\mathcal{H}^n|_E$ at $x$ are supported on biLipschitz images of $\mathbb{R}^n$.

In this article we show that, for 1-dimensional sets, this sufficient condition for rectifiability can be weakened further to tangents that are supported on \emph{connected} metric spaces.

Let $\mathcal{C}^*$ be the space of pointed metric measure spaces $(X,d,\mu,x)$ for which the support of $\mu$ is connected.

\begin{theorem}
	\label{lovelytheorem}
	Let $(X,d)$ be a complete metric space and let $E\subset X$ be $\mathcal{H}^1$-measurable with $\mathcal{H}^1(E)<\infty$.
	The following are equivalent:
	\begin{enumerate}
		\item $E$ is $1$-rectifiable.
		\item For $\mathcal{H}^1$-a.e. $x\in E$, $\Theta_*^1(E,x)>0$ and
		\[\textnormal{Tan}(X,d,\mathcal{H}^1\llcorner E,x)=
		\{(\mathbb R,|\cdot|,\mathcal{H}^1/2,0)\}.\]
		\item For $\mathcal{H}^1$-a.e. $x\in E$, $\Theta_*^1(E,x)>0$ and
		\[\textnormal{Tan}(X,d,\mathcal{H}^1\llcorner E,x)\subset \mathcal{C}^*.\]
	\end{enumerate}
\end{theorem}
The fact that 1) implies 2) is given in \cite[Theorem 6.6]{bate} by applying the results of Kirchheim \cite{paper:kirchheim}.
Since 2) implies 3) is immediate, our attention is focused on the implication 3) implies 1).
We emphasise that, even in the case that all tangent spaces are supported on $\mathbb R$, our proof follows a completely different approach than that of \cite{bate}.

Indeed, the geometric idea of our proof begins with a result of Besicovitch \cite[Lemma 9]{paper:besi2} that covers a purely 1-unrectifiable subset of $\mathbb R^2$ by a large number of \emph{circle pairs}.
A circle pair of a set $F$ is simply $B(x,|x-y|)\cup B(y,|x-y|)$, where $x$ and $y$ are points in $F$.
Besicovitch demonstrates that for a $1$-unrectifiable set $F\subset \mathbb R^2$ and any $r>0$, one may find disjoint circle pairs, whose radii are at most $r$ and sum to at least $\mathcal H^1(F)/9$, such that the common regions are disjoint from $F$:
\[F\cap B(x,|x-y|)\cap B(y,|x-y|)=\emptyset.\]
This construction can be used to show that $1$-unrectifiable sets in $\mathbb{R}^2$ (with finite Hausdorff measure and positive lower density) cannot have tangent measures that are supported on lines:
any such tangent is forced to be close to both $x$ and $y$ and hence must cross the empty region, thereby contradicting the idea that a tangent must stay close to the set it is approximating. 

There are two issues in extending this idea to prove Theorem \ref{lovelytheorem}.
Firstly, his proof only works in $\mathbb{R}^2$, in particular it uses the fact that boundaries of balls are continua. Secondly, when considering tangents supported on connected sets (even in $\mathbb{R}^2$), Besicovitch's circle pairs are no longer sufficient: a \emph{connected} tangent can lie close to both $x$ and $y$ and wander outside $B(x,|x-y|)\cup B(y,|x-y|)$ and thereby avoid the common region.

To solve this, we strengthen Besicovitch's notion of circle pairs to what we term Besicovitch \emph{partitions}, see Definition \ref{def:besipart}. Roughly speaking, these partition $F$ into two closed sets, and we consider points $x,y$ that minimise the distance between the components.
In Theorem \ref{lemma:besicovtichpartitions} we show that in any complete metric space, for any compact purely $1$-unrectifiable set $F\subset X$ and $R,r>0$, there exist Besicovitch partitions with $\sum d(x_i,y_i) \sim \mathcal H^1(F)$ such that $d(x_i,y_i)<r$ and moreover the balls $B(x_i,Rd(x_i,y_i))$ (for different $i$) are pairwise disjoint.
On the other hand, in Corollary \ref{decomposemebaby}, we show that if all tangents of a space are connected, then in fact there exists a set of large measure and an $R>0$ such that all tangents to points in this set are connected in a ball of radius $R$.
Thus, we again arrive to the desired contradiction.

To conclude the introduction we mention the related work of Preiss and Ti{\v{s}}er \cite{zbMATH00120128} that gives a different adaptation of Besicovitch's circle pair argument (to metric spaces) in order to make progress on the Besicovitch one-half conjecture.
These results, and the progress towards the conjecture, were recently extended further by De Lellis, Glaudo, Massaccesi and Vittone \cite{2404.17536}. Further, our result can be viewed as a substantial strengthening of a theorem of Azzam and Mourgoglou \cite{azzam}, which states that a doubling metric measure space with positive 1-density and connected support is 1-rectifiable.

The structure of the paper is as follows.
After recalling definitions and fixing some notation in \S \ref{prelims}, we present the definition of a metric tangent and recall various properties from \cite{bate}. Theory that is specific to connected tangents is developed in \S \ref{connectedtangents}.

In \S \ref{besicovitchpartitions} we prove our strengthened version of Besicovitch's circle pair result. Finally, in Theorem \ref{ultimate} in \S \ref{results}, we combine all of our ideas and perform the contradiction argument described above.
This gives the final piece required to prove Theorem \ref{lovelytheorem}.

	\section{Preliminaries}\label{prelims}
	In all the following, let $(X,d)$ be a metric space. For $E\subset X$ and $a\in X$, let
\[\textnormal{dist}(a,E):=\inf\{(a,x):x\in E\}.\]
Set
\[\textnormal{diam}(E):=\sup\{d(x,y):x,y\in E\}.\]
We use $B(x,r)$ to denote the closed ball centred at $x$ and of radius $r>0$ and $U(x,r)$ to denote the corresponding open ball.
For $\varepsilon>0$, we will denote the $\varepsilon$-neighbourhood of $E$ via an underlined subscript
\[E_{\boldsymbol{\underline\varepsilon}}:=\{x\in X:\textnormal{dist}(x,E)\leq \varepsilon\}.\]
For two sets $A,B\subset X$, we define their Hausdorff distance as
\[d_H(A,B)=\inf\{\varepsilon>0:A\subset B_{\boldsymbol{\underline\varepsilon}}\text{ and }B\subset A_{\boldsymbol{\underline\varepsilon}}\}.\]

\subsection{General Measure Theory}
 
Let $\mathcal{M}(X)$ denote the set of all Borel regular measures on $X$ and $\mathcal{M}_{\text{loc}}(X)\subset\mathcal{M}(X)$ be the set of all those that are finite on all bounded sets.
 
The support of a measure $\mu\in\mathcal{M}(X)$ is a subset of $X$ and is defined as
\[\supp\mu:=\{x\in X: \text{ for all open neighbourhoods }U_x\text{ of }x,\,\mu(U_x)>0\}.\]
We may restrict a measure to a set $E\subset X$ defined as $\mu\llcorner E(A):=\mu(E\cap A)$ and pushforward a measure by a map $f$ defined as $f_\#\mu(A):=\mu(f^{-1}(A))$.
%We also require the following standard fact:
%\begin{lemma}
%	Let $\mu\in\mathcal{M}_{loc}(X)$, then for any $r>0$,
%	\[x\mapsto \mu(U(x,r)) \text{ is lower semicontinuous}\]
%	and 
%	\[x\mapsto \mu(B(x,r)) \text{ is upper semicontinuous.}\]
%\end{lemma}
\subsection{Rectifiability}

\begin{definition}[Hausdorff measure]\label{def:hausdorff}
	For $s\geq 0$, $\delta\in[0,\infty]$ and a set $E\subset X$, let
	\[\mathcal{H}_\delta^s(E)=\inf\left\{\sum_{i=1}^\infty \textnormal{diam}(A_i)^s:\,E\subset\bigcup A_i,\,\textnormal{diam}(A_i)<\delta\right\}.\]
	Then the $s$-dimensional Hausdorff measure of $E$ is defined as
	\[\mathcal{H}^s(E):=\lim_{\delta\rightarrow 0}\mathcal{H}^s_\delta(E).\]
\end{definition}
For any $s$, Hausdorff measure is a Borel regular measure \cite[Corollary 4.5]{book:mattila}.
\begin{lemma}\label{lemma:hausdorff}
	Let $E\subset X$ be $\mathcal{H}^s$-measurable with $\mathcal{H}^s(E)<\infty$. Then for any $\varepsilon>0$ we can find a $\delta_0>0$ such that for any collection of sets of the form $\{A_i\subset X:\textnormal{diam}(A_i)\leq \delta_0\}_{i\in\mathbb{N}}$ satisfies
	\[\mathcal{H}^s(E\cap \bigcup_{i\in\mathbb{N}} A_i)\leq \sum_i \textnormal{diam}(A_i)^s+\varepsilon.\] 
\end{lemma}
\begin{proof}
	By the definition of Hausdorff measure, $\mathcal{H}^s_\delta(E)$ is monotonically increasing as $\delta\rightarrow 0$. By assumption, $\mathcal{H}^s(E)<\infty$ hence for any $\varepsilon>0$ we can find a $\delta_0>0$ such that
	\[\mathcal{H}^s(E)\leq\mathcal{H}^s_{\delta}(E)+\varepsilon \qquad\text{for all }\delta\leq\delta_0.\]
	Let $\{A_i\subset X:\textnormal{diam}(A_i)\leq \delta_0\}_{i\in\mathbb{N}}$, then
	\begin{align*}
		\mathcal{H}^s(E\cap \bigcup_{i\in\mathbb{N}} A_i)+\mathcal{H}^s(E\backslash \bigcup_{i\in\mathbb{N}} A_i)&=\mathcal{H}^s(E)\\
		&\leq\mathcal{H}^s_{\delta}(E)+\varepsilon \\
		&\leq\mathcal{H}^s_\delta(E\backslash \bigcup_{i\in\mathbb{N}} A_i)+\mathcal{H}^s_\delta(E\cap \bigcup_{i\in\mathbb{N}} A_i)+\varepsilon \\
		&\leq \mathcal{H}^s(E\backslash \bigcup_{i\in\mathbb{N}} A_i)+\sum_i \textnormal{diam}(A_i)^s+\varepsilon.
	\end{align*}
	Subtracting $\mathcal{H}^s(E\backslash \bigcup_{i\in\mathbb{N}} A_i)$ from both sides (which is finite by assumption) gives the desired result.
\end{proof}

\begin{definition}[Lipschitz Map]\label{def:lipschitz}
	Given a metric space $ (X,d) $ and a set $ A\subset \mathbb{R}^n $, a map $ f:A\rightarrow X $ is Lipschitz if there exists $ K\geq0 $ such that
	\[ d(f(x),f(y))\leq K|x-y|,\qquad \forall x,\,y\in A. \]
\end{definition}

\begin{definition}[Rectifiable]\label{def:rectifiable}
	A set $ E\subset X $ is $n$-rectifiable if there are at most countably many Lipschitz maps $ f_i:A_i\rightarrow X $, with $ A_i\subset \mathbb{R}^n $, such that 
	\[ \mathcal{H}^n(E\setminus \bigcup_i f_i(A_i) )=0. \]
	A set $ P\subset X $ is purely $n$-unrectifiable if $ \mathcal{H}^n(P\cap E)=0 $ for all $ n $-rectifiable $ E\subset X. $
\end{definition}

\begin{lemma}\label{lemma:decomprec/unrec}
	Let $E\subset X$ be $\mathcal{H}^n$-measurable with $\mathcal{H}^n(E)<\infty$. There exists a decomposition $E=R\cup P$ where $R$ is $n$-rectifiable and $P$ is purely $n$-unrectifiable.
\end{lemma}
\begin{proof}
	The proof for $X=\mathbb{R}^m$ is given in \cite[Theorem 15.6]{book:mattila}.
	The proof for general $X$ is identical.
\end{proof}

\subsection{Density and Doubling}
\begin{definition}[Density]\label{def:density}
	For $\mu\in\mathcal{M}(X)$, positive number $ s\in[0,\infty) $ and point $ x\in X $, we can define the upper and lower $ s $-dimensional Hausdorff densities of $ \mu $ at $ x $ as
	\begin{equation*}
		\Theta^{*s}(\mu,x)=\limsup_{r\rightarrow 0} \frac{\mu(B(x,r))}{(2r)^s};
	\end{equation*}
	\begin{equation*}
		\Theta^{s}_*(\mu,x)=\liminf_{r\rightarrow 0} \frac{\mu(E\cap B(x,r))}{(2r)^s}.
	\end{equation*}
	If the upper and lower densities agree then the common value is the $ s $-dimensional density of $ \mu $ at $ x $ and is denoted $ \Theta^s(\mu,x). $
	If $ E\subset X $ we write $\Theta^{*s}(E,x)$, $\Theta^s_*(E,x)$ and $\Theta^s(E,x)$ for $\Theta^{*s}(\mathcal H^n|_E,x)$, $\Theta^{s}_*(\mathcal H^n|_E,x)$ and $\Theta^*(\mathcal H^n|_E,x)$ respectively.
\end{definition}

\begin{theorem}[\cite{book:mattila} Theorem 6.2]\label{theorem:density}
	Let $ E\subset X $ be such that $ 0<\mathcal{H}^s(E)<\infty $ then
	\begin{enumerate}[left=1cm, label= \textbf{(\roman{enumi})}, ref=Theorem  \ref{theorem:density}\textbf{(\roman{enumi})}]
		\item $ 2^{-s}\leq\Theta^{*s}(E,x)\leq 1 $ for $ \mathcal{H}^s $-almost all $ x\in E $;\label{theorem:densitypart1}
		\item If $ E $ is $ \mathcal{H}^s $-measurable, then $ \Theta^{*s}(E,x)=0 $ for $ \mathcal{H}^s $-almost all $x\in X\backslash E. $\label{theorem:densitypart2}
	\end{enumerate}
\end{theorem}

\begin{corollary}\label{cor:densityofrestrictedsets}
	Let $ E,F\subset X $ be $\mathcal{H}^n$-measurable with $F\subset E$ and $\mathcal{H}^n(E)<\infty$. Then for $\mathcal{H}^n$-a.e. $x\in F$,
	\[\Theta^{*n}(F,x)=\Theta^{*n}(E,x)\qquad\text{and}\qquad \Theta^{n}_*(F,x)=\Theta^{n}_*(E,x).\]
\end{corollary}

\begin{definition}[Ahlfors $n$-regular measures]
	A metric space $X$ along with a measure $\mu\in\mathcal{M}_{loc}(X)$ is $C$-Ahlfors $n$-regular if for all $r>0$,
	\[\frac{1}{C}r^n\leq \mu(B(x,r))\leq Cr^n\qquad \forall\,x\in \supp\mu.\]
	Denote the set of all $C$-Ahlfors regular metric measure spaces as $\mathcal{A}^n(C)$.
\end{definition}

Finally we have the notion of doubling for measures.
\begin{definition}[Doubling]
	For $K\geq1$, $\mu\in\mathcal{M}(X)$ is $K$-doubling if 
	\[0<\mu(B(x,\,2r))\leq K\mu(B(x,\,r))<\infty,\qquad\forall\,x\in X,\,r>0,\]
	or asymptotically doubling if
	\[\limsup_{r\rightarrow 0}\frac{\mu(B(x,\,2r))}{\mu(B(x,\,r))}<\infty\qquad\textnormal{ for }\mu-a.e.\,x\in X.\]
\end{definition}

\begin{remark}\label{remark:densedouble}
	Let $E\subset X$ be $\mathcal{H}^n$-measurable with $\mathcal{H}^n(E)<\infty$ and $\Theta^n_*(E,x)>0$ for $\mathcal{H}^n$-a.e. $x\in E$. Then $\mathcal{H}^n\llcorner E$ is asymptotically doubling. Indeed, By \ref{theorem:densitypart1}, for $\mathcal{H}^n$-a.e. $x\in E$ we have $\Theta^{*n}(E,x)\leq 1$, so for any such $x$ with $\Theta^n_*(E,x)>0$ we have:
	\begin{align*}
		\limsup_{r\rightarrow 0}\frac{\mathcal{H}^n(E\cap B(x,2r))}{\mathcal{H}^n(E\cap B(x,r))} &= 2^n\limsup_{r\rightarrow 0}\frac{\mathcal{H}^n(E\cap B(x,2r))}{(2\cdot 2r)^n}\cdot \frac{(2r)^n}{\mathcal{H}^n(E\cap B(x,r))}\\
		&\leq \frac{2^n}{\Theta_*^n(E,x)}<\infty.
	\end{align*}
\end{remark}

\begin{theorem}[\cite{book:heinonen} Section 3.4]\label{lebesguepoints}
	Let $\mu\in \mathcal{M}(X)$ be asymptotically doubling and $E\subset X$. Then for $\mu$-a.e. $x\in E$
	\begin{equation*}
		\lim_{r\rightarrow 0}\frac{\mu(E\cap B(x,r))}{\mu(B(x,r))}=1.
	\end{equation*}
	Any such $x\in S$ is known as a Lebesgue density point of $E$.
\end{theorem}

	\subsection{Pointed Gromov-Hausdorff distance}\label{metricspaces}
	To develop a theory of metric tangents, we require a few standard tools from metric geometry.
%\begin{definition}[Hausdorff distance] \unsure{Is this even necessary?}
%	Let $ (X,d) $ be a metric space and $ A,\,B\subset X $. Then we define the Hausdorff distance between $ A $ and $ B $ as
%	\begin{equation}
%		d_{H}(A,\,B):=\inf\{\varepsilon>0: A\subset B_{\underline\varepsilon} \textnormal{ and } B\subset A_{\underline\varepsilon}\}.
%	\end{equation}
%\end{definition}
%\begin{definition}[Gromov-Hausdorff distance] \unsure{Is this even necessary?}
%	Let $ (X,\,d_x) $ and $ (Y,\,d_y) $ be metric spaces, then we can define the Gromov-Hausdorff distance as
%	\[ d_{GH}(X,Y):=\inf\{d_{H}(X,\,Y):\exists \textnormal{ isometric embeddings }X,\,Y\rightarrow Z \}. \]
%\end{definition}
\begin{definition}[Pointed metric (measure) space]
	A pointed metric space $ (X,\,d,\,x) $ is an ordered triple consisting of a metric space $(X,\,d) $ and a point $ x\in X $. A pointed metric measure space $ (X,\,d,\,\mu,\,x) $ consists of a complete and separable pointed metric space $ (X,x) $ and a measure $ \mu\in\mathcal{M}_{loc}(X)$.
	 
	An isometric embedding $ (X,\,d,\,x)\rightarrow(Y,\,\rho,\,y) $ of pointed metric spaces is an isometric embedding $ (X,\,d) \rightarrow (Y,\,\rho) $ such that $ x\mapsto y. $
\end{definition}
\begin{notation}
	When the meaning is clear, we will omit the metric and abbreviate $(X,d,x) $ or $ (X,d,\mu,x) $ to $(X,x) $ or $ (X,\mu,x) $ respectively.
\end{notation}
We first require a pointed notion of distance for sets that share the same ambient space.
\begin{definition}[Pointed Hausdorff distance]
	Let $ (Z,d) $ be a metric space and $ X,\,Y\subset Z $. Then for $ z\in Z $, we define the pointed Hausdorff distance between $ X $ and $ Y $ as
	\begin{equation}
		d_{H(z)}(X,\,Y):=\inf\left\{\varepsilon>0: X\cap B(z,\,1/\varepsilon)\subset Y_{\underline\varepsilon},\, Y\cap B(z,\,1/\varepsilon)\subset X_{\underline\varepsilon} \right\}.
	\end{equation}
\end{definition}
We can now define the pointed Gromov-Hausdorff distance which allows us to compare sets across different ambient metric spaces.
\begin{notation}
	In general we shall not distinguish between the original objects $X$ and $\mu$ and their isometric images $\iota(X)\subset Z$ and $\iota_\#\mu$.
\end{notation}
\begin{definition}[Pointed Gromov-Hausdorff distance] 
	Let $ (X,\,x) $ and $ (Y,\,y) $ be pointed metric spaces. We define the pointed Gromov-Hausdorff distance as
	\[ d_{pGH}((X,\,x),(Y,\,y))=\inf\left\{d_{H(z)}(X,\,Y):\parbox{4.2cm}{\centering $\exists$ metric space $(Z,\,z)$ and isometric embeddings\\ $(X,\,x),(Y,\,y)\rightarrow(Z,\,z)$}\right\}. \]
\end{definition}
Finally, let us define the space on which $d_{pGH}$ is a metric. Recall that a metric space is proper if all closed balls are compact.
\begin{corollary}[Corollary 2.25 \cite{bate}]
	If $ \mathbb{M}_p $ is the set of isometry classes of proper pointed metric spaces, then $ (\mathbb{M}_p,\,d_{pGH}) $ is a complete and seperable metric space. In particular, if $ (X,\,x) $ and $ (Y,\,y) $ are proper with 
	\[ d_{pGH}((X,x),(Y,y))=0 \]
	then there is an isometry $(X,x)\rightarrow(Y,y)$.
\end{corollary}

	\subsection{Convergence of measures}
Let $C_{bs}(X)$ denote the space of bounded and continuous maps $f:X\to \mathbb{R}$ with bounded support.
\begin{definition}[Weak* convergence]
	A sequence $\mu_i\in\mathcal{M}_{loc}(X)$ weak* converges to $\mu\in\mathcal{M}_{loc}(X)$ if
	\[\int f\,d\mu_i\rightarrow \int f\,d\mu\]
	for all $f\in C_{bs}(X)$. We use $\xrightharpoonup{*}$ to denote this convergence.
\end{definition}
\begin{prop}[\cite{book:maggi} Proposition 4.26]\label{prop:properties_of_weak_con}
	Let $\{\mu_i\},\mu\in\mathcal{M}_{loc}(X)$. The following are equivalent:
	\begin{enumerate}
		\item $\mu_i\xrightharpoonup{*}\mu$.
		\item If $x_i\rightarrow x\in X$, then for every $r>0$,
		\begin{equation}\label{balsam}
			\mu(U(x,r))\leq\liminf_{i\rightarrow \infty}\mu_i(U(x_i,r))
		\end{equation}
		and 
		\begin{equation}\label{fir}
			\mu(B(x,r))\geq\limsup_{i\rightarrow \infty}\mu_i(B(x_i,r)).
		\end{equation}
	\item If $B(x,r)$ is such that with $\mu(B(x,r)\backslash U(x,r))=0$, then
	\[\mu(B(x,r))=\lim_{i\rightarrow \infty}\mu_i(B(x_i,r)).\]
	\end{enumerate}
	Moreover, if $\mu_i\xrightharpoonup{*}\mu$, then for every $x\in\supp\mu$, there exists a sequence $x_i\in \supp\mu_i$ such that $x_i\rightarrow x$.
\end{prop}

In general it is not true that if there exists a sequence $\mu_i$ with $x_i\in \supp\mu_i$ such that $\mu_i\xrightharpoonup{*}\mu$ and $x_i\rightarrow x$ then $x\in\supp\mu$. However, we do have this for Ahlfors regular measures.
\begin{prop}\label{prop:properties_of_weak_con_ahlfors}
	Let $\mu_i\in\mathcal{A}^n(C)$ in a metric space $X$ be such that $\mu_i\xrightharpoonup{*}\mu$. Then $\mu\in\mathcal{A}^n(C)$ and if $x_i\in\supp\mu_i$ is such that $x_i\rightarrow x$, then $x\in \supp\mu$.
\end{prop}
\begin{proof}
	By Proposition \ref{prop:properties_of_weak_con}, for any $x\in \supp\mu$ there exists a sequence $x_i\in \supp\mu_i$ such that $x_i\rightarrow x$. Thus for any $\varepsilon>0$ we have
	\begin{equation*}
		\mu(B(x,r))\leq \mu(U(x,r+\varepsilon))\leq \liminf_{i\rightarrow \infty}\mu_i(B(x_i,r+\varepsilon))\leq C(r+\varepsilon)^n
	\end{equation*}
	and
	\begin{equation*}
		\mu(B(x,r))\geq \limsup_{i\rightarrow \infty}\mu_i(B(x_i,r))\geq r^n/C.
	\end{equation*}
	Letting $\varepsilon\rightarrow 0$ demonstrates that $\mu\in \mathcal{A}^n(C)$. 
	 
	For the second claim, fix $r>0$ and let $i_0>0$ be sufficiently large that $B(x_i,r)\subset B(x,2r)$ for all $i\geq i_0$. Thus
	\begin{equation*}
		\mu(B(x,2r))\geq \limsup_{i\rightarrow \infty}\mu_i(B(x,2r))\geq \limsup_{i\rightarrow \infty}\mu_i(B(x_i,r))\geq r^n/C.
	\end{equation*}
	By the arbitrariness of $r$, we have demonstrated that $x\in\supp \mu$.
\end{proof}
As previously intimated, we will be working in terms of tangent measures as opposed to sets.
\begin{definition}\label{def:measuremetric}
	Let $x\in X$ and $\mu,\nu\in\mathcal{M}(X)$. For $L,\,r>0$ define
	\[F^{L,\,r}_x(\mu,\nu):=\sup\left\{\int f\,d(\mu-\nu):f\in\textnormal{Lip}_L(X;[-1,1])\,\textnormal{ with }\supp f\subset B(x,\,r) \right\}.\]
	Further, where the infimum exists let
	\[F_x(\mu,\nu):=\inf\left\{\varepsilon\in(0,1/2): F_x^{1/\varepsilon,1/\varepsilon}(\mu,\nu)<\varepsilon\right\}\]
	and if no such infimum exists, set $F_x(\mu,\nu)=1/2$. 
\end{definition}
%\begin{lemma}[Lemma 2.14 \cite{bate}]
%	$F_x$ is a metric on $\mathcal{M}(X)$. If $X$ is complete and separable then so is $(\mathcal{M}_{loc}(X),F_x)$.
%\end{lemma}
We have the following useful equivalence:
\begin{prop}[Proposition 2.13 \cite{bate}]\label{prop:Fequivweak*}
	Convergence in $F_x$ implies weak* convergence. If $X$ is complete and separable then the converse is also true.
\end{prop}
We now introduce a measured version of $d_{pGH}$.
\begin{definition}
	For pointed metric measure spaces $(X,\mu,z)$ and $(Y,\nu,y)$, define
	\begin{multline*} d_{pmGH}((X,\mu,z),\,(Y,\nu,y))=\\
		\inf\left\{d_{H(z)}(X,Y)+F_z(\mu,\,\nu):\parbox{4.2cm}{\centering $\exists$ metric space $(Z,\,z)$ and isometric embeddings\\ $(X,\,x),(Y,\,y)\rightarrow(Z,\,z)$}\right\}.
	\end{multline*}
\end{definition}
This is a complete and separable metric on the set of isometry classes of pointed metric measure spaces $\mathbb{M}_{pm}$, see \cite[Corollary 4.6]{bate}. We are now ready to introduce the metric that will be used to define tangent measures in the next section.

\begin{definition}
	Let $(X,\mu,z)$ and $(Y,\nu,y)$ be pointed metric measure spaces and define
	\[d_*((X,\mu,z),\,(Y,\nu,y))=\inf\left\{F_z(\mu,\,\nu):\parbox{4.2cm}{\centering $\exists$ metric space $(Z,\,z)$ and isometric embeddings\\ $(X,\,x),(Y,\,y)\rightarrow(Z,\,z)$}\right\}.\]
\end{definition}
Let $\mathbb{M}_*$ be the set of pointed metric measure spaces under the equivalence relation 
\[(X,\mu,x)\sim (Y,\nu,y) \text{ if }(\supp\mu\cup \{x\},\mu,x)\textnormal{ is isometric to }(\supp\nu\cup\{y\},\nu,y).\]
Then as shown in \cite[Corollary 4.13]{bate}, $(\mathbb{M}_*,d_*)$ is a complete and separable metric space. In light of this, to simplify notation we will sometimes write $(\mu,x)$ in place of $(X,\mu,x)$.
 
The final piece of the puzzle will be to extract some relationship between the convergence of measures and the convergence of their supports. Although $d_*$ and $d_{pmGH}$ are not equivalent, we do have a strong relationship as outlined in the following key result.

\begin{theorem}[Theorem 4.11 \cite{bate}]\label{theorem:largesubsets}
	For $\varepsilon>0$, let $(X,\mu,x),(Y,\nu,y)\in\mathbb{M}_{*}$ be pointed metric measure spaces with $x\in\supp\mu$, $y\in \supp\nu$ and
	\[d_*((X,\mu,x),(Y,\nu,y))<\varepsilon.\]
	There exists compact
	$x\in K_\mu\subset\supp\mu\cap B(x,1/\varepsilon)$ and $y\in K_\nu \subset \supp\nu\cap B(y,1/\varepsilon)$
	such that 
	\begin{enumerate}[left=1cm, label=\textbf{M\arabic*.},ref=\textbf{M\arabic*.}]
		\item $\max\{\mu(B(x,1/\varepsilon)\backslash K_\mu),\nu(B(y,1/\varepsilon)\backslash K_\nu)\}<\varepsilon$;\label{largesubsets1}
		\item $d_{pmGH}((K_\mu,\mu\llcorner K_\mu,x),(K_\nu,\nu\llcorner K_\nu,y))<3\varepsilon$.\label{largesubsets2}
	\end{enumerate}
	Conversely, if there exist compact $K_\mu,\,K_\nu$ as above satisfying \ref{largesubsets1} with
	\[d_{pmGH}((K_\mu,\mu\llcorner K_\mu,x),(K_\nu,\nu\llcorner K_\nu,y))<\varepsilon\]
	then
	\[d_*((X,\mu,x),(Y,\nu,y))<3\varepsilon.\]
	That is, $d_*$ is bi-Lipschitz equivalent to the metric defined by taking the infimum over all $\varepsilon\in(0,1/2)$ such that \ref{largesubsets1} and \ref{largesubsets2} hold.
\end{theorem}
We conclude this section with the following proposition that allows us to work entirely in some common space $(Z,z)$.
\begin{prop}[Proposition 2.20 \cite{bate}]\label{prop:embeddingsandconvergence}
	Let $\{(X_i,\mu_i,x_i)\}_i$ be a sequence of pointed metric measure spaces. We have that
	\[d_*((X_i,\mu_i,x_i),(X,\mu,x))\rightarrow 0,\quad \text{respectively}\quad d_{pmGH}((X_i,\mu_i,x_i),(X,\mu,x))\rightarrow 0,\]
	for some pointed metric measure space $(X,\mu,x)$ if and only if there exists a complete and separable pointed metric space $(Z,z)$ and isometric embeddings $(X_i,x_i)\rightarrow (Z,z)$ and $(X,x)\rightarrow(Z,z)$ such that
	\[F_z(\mu_i,\mu)\rightarrow 0,\qquad\text{or respectively}\qquad d_{H(z)}(X,Y)+F_z(\mu,\,\nu)\rightarrow 0.\]
\end{prop}

%\begin{prop}\label{prop:pseudometric}
%	For any $a,b>0$, $G_{a,b}$ is a complete pseudometric on the class of pointed metric measure spaces. Moreover,
%	\[G_{a,b}((X_i,\mu_i,x_i),(X,\mu,x))\rightarrow 0\]
%	if and only if there exists a complete and separable pointed metric space $(Z,z)$ and isometric embeddings $(X_i,x_i)\rightarrow(Z,z)$ and $(X,x)\rightarrow(Z,z)$ such that
%	\[aH_z(X_i,X)+bF_z(\mu_i,\mu)\rightarrow 0.\]
%\end{prop}

	\subsection{Tangents of metric measure spaces}\label{metrictangents}
	Tangent measures were first introduced by Preiss \cite{paper:preiss} in Euclidean space as a natural setting in which to explore questions around density. We refer the reader to \cite[Page 184]{book:mattila} for an introduction to tangent measures. With the notion of convergence of metric measure spaces established in the previous section, the definition of tangent measures in metric spaces follows analogously. First, let $(\mu,d,x)\in \mathbb{M}_*$ with $x\in\supp\mu$ and for $r>0$, define
\[T_r(\mu,d,x):=\left(\frac{\mu}{\mu(B(x,r))},\frac{d}{r},x\right).\]
\begin{definition}[Tangent measure]
	For $(\mu,d,x)\in \mathbb{M}_*$, we say that $(\nu,\rho,y)\in \mathbb{M}_*$ is a tangent measure if there exists a sequence $\{r_i\}$ with $r_i\rightarrow 0$ such that
	\[d_*(T_{r_i}(\mu,d,x),(\nu,\rho,y))\rightarrow 0 \qquad \text{as }i\rightarrow \infty.\]
	The set of all tangent measures for $(\mu,d,x)$ is denoted $\textnormal{Tan}(\mu,d,x)$.
\end{definition}
Note that this definition of a tangent measure allows for different sequences of radii to produce different tangents - in a sense allowing the tangents to ``rotate''. A basic example is that of a spiral in the plane that converges very quickly in the radial direction. In this case, the set of tangent measures at the origin consists of all half lines, equipped with Lebesgue measure.

\begin{corollary}[Corollary 4.7 \cite{bate}]\label{cor:restrict}
	Let $(\mu,d,x)\in\mathbb{M}_*$ and let $S\subset\supp\mu$ be $\mu$-measurable with $x$ a Lebesgue density point of $S$. Then
	\[\textnormal{Tan}(\mu\llcorner S,x)=\textnormal{Tan}(\mu,x).\]
\end{corollary}
We also have the following compactness result.
\begin{prop}[Proposition 5.3 \cite{bate}]\label{prop:tancompact}
	Let $X$ be a metric space and let $\mu\in\mathcal{M}_{loc}(X)$ be asymptotically doubling. For $\mu$-a.e. $x\in\supp\mu$ and every $r_0>0,$
	\[\{T_r(\mu,x):0<r<r_0\} \textnormal{ is pre-compact.}\]
	In particular, for any $x\in\supp\mu$ satisfying the above, $\textnormal{Tan}(\mu,x)$ is a non-empty compact metric space when equipped with $d_*$ and 
	\[\forall\, \delta>0,\exists\, r_x>0:d_*(T_r(\mu,x),\textnormal{Tan}(\mu,x))\leq \delta\,\forall\, r\in (0,r_x).\]
\end{prop}
Due to the rescaling of the metric in the definition and the nature of embeddings we can derive the fact that the tangent space is closed under rescaling of metrics and translation of points. 
\begin{lemma}\label{lemma:rescalemetrics}
	Let $(X,d,\mu)$ be a metric measure space such that $\mu$ is asymptotically doubling.
	\begin{enumerate}[left=1cm, label= \textbf{(\arabic*)}, ref=Lemma \ref{lemma:rescalemetrics}\textbf{(\arabic*)}]
		\item  For $\mu$-a.e. $x\in\supp\mu$ and $s>0$, there exists $C>0$ such that
		\[(\nu,\rho,y)\in \textnormal{Tan}(\mu,x) \quad \implies \quad (C\nu,s\rho,y)\in\textnormal{Tan}(\mu,x).\]
		\label{lemma:rescalemetrics1}
		\item  For $\mu$-a.e. $x\in\supp\mu$,
		\[(\nu,y)\in \textnormal{Tan}(\mu,x) \textnormal{ and } z\in\supp\nu \quad \implies \quad (\nu,z)\in\textnormal{Tan}(\mu,x).\] \label{lemma:rescalemetrics2}
	\end{enumerate}
\end{lemma}
\begin{proof}
	Proof of \textbf{(2)} may be found in \cite[Proposition 5.5]{bate} so we just present a proof of \textbf{(1)}. If $(\nu,\rho,y)\in\textnormal{Tan}(\mu,x)$ then there exists $r_i\rightarrow 0$ such that \[d_*(T_{r_i}(\mu,d,x),(\nu,\rho,y))\rightarrow 0.\] By Proposition \ref{prop:embeddingsandconvergence} there exists a complete and separable metric space $(Z,\lambda,z)$ and  isometries
	\[\left(\supp\frac{\mu}{\mu(B_d(x,r_i))},\frac{d}{r_i},x\right),(\supp\nu,\rho,y)\rightarrow (Z,\lambda,z)\] 
	such that
	\begin{equation*}
		F_z\left(\frac{\mu}{\mu(B_d(x,r_i))},\nu\right)\rightarrow 0.
	\end{equation*}
	Let
	\[C_i= \frac{\mu(B_{d}(x,r_i))}{\mu(B_{sd}(x,r_i))}.\]
	Since $\mu$ is asymptotically doubling we can set (up to subsequence)
	\[C=\lim_{i\to\infty}C_i>0.\]
	Consider now the sequence $r_i/s\rightarrow 0$ and note that
	\begin{align*}
		T_{r_i/s}(\mu,d,x)&=\left(\frac{\mu}{\mu(B_d(x,r_i/s))},\frac{sd}{r_i},x\right)\\
		&=\left(\frac{\mu}{\mu(B_{sd}(x,r_i))},\frac{sd}{r_i},x\right)\\
		&=T_{r_i}(\mu,sd,x).
	\end{align*}
	Here the subscript denotes the metric that defines the ball.
	By pre and post-scaling the previous isometries by $s$, we can construct new isometries 
	\[\left(\supp\frac{\mu}{\mu(B_{sd}(x,r_i))},\frac{sd}{r_i},x\right),(\supp\nu,s\rho,y)\rightarrow (Z,s\lambda,z).\] 
	We will write $\mu^s$ and $\nu^s$ for the pushforward of $\mu,\,\nu$ under these isometries in $(Z,s\lambda)$
	and scaling $\nu^s$ we have,
	\begin{align*}
		d_*\left(T_{r_i}(\mu,sd,x),\left(C\nu,s\rho,y\right)\right)&\leq F_z\left(\frac{\mu^s}{\mu(B_{sd}(x,r_i))},C\nu^{s}\right)\\
		&\leq F_z\left(\frac{\mu^s}{\mu(B_{sd}(x,r_i))},C_i\nu^{s}\right)+F_z\left(C_i\nu^{s},C\nu^{s}\right)\\
		&\leq F_z\left(\frac{\mu^s}{\mu(B_{sd}(x,r_i))},C_i\nu^{s}\right)+|C-C_i|.
	\end{align*}
	Recall that by Proposition \ref{prop:Fequivweak*} we have that convergence in $F_z$ is equivalent to weak* convergence. Thus for any $f\in C_{bs}(Z)$,
	 \begin{align*}
		\int f(z)\, d\,\left(\frac{\mu^s}{\mu(B_{sd}(x,r_i))}-C_i\nu^{s}\right)(z)=C_i\int f(sz)\, d\,\left(\frac{\mu}{\mu(B_{d}(x,r_i))}-\nu\right)(z)\rightarrow 0
	\end{align*}
	which combined with $|C-C_i|\to 0$ concludes the proof.
\end{proof}
The upper and lower density of the measure at a point has implications on the regularity of tangents at that point.
\begin{lemma}\label{lemma:ahlforsregular}
	For $(\mu, x)\in\mathbb{M}_*$ with $0<\Theta^s_*(\mu,x),\,\Theta^{s*}(\mu,x)<\infty$, we have that
	\[\frac{\Theta^s_*(\mu,x)}{\Theta^{s*}(\mu,x)} r^s\leq \nu(B(y,r))\leq \frac{\Theta^{s*}(\mu,x)}{\Theta^s_*(\mu,x)}  r^s\qquad\forall (\nu,y)\in\textnormal{Tan}(\mu,x).\]
	In particular, if $0<\Theta^s_*(\mu,x),\,\Theta^{s*}(\mu,x)<\infty$ for $\mu$-a.e. $x\in\supp\mu$, then for $\mu$-a.e. $x\in \supp\mu$,
	\[\textnormal{Tan}(\mu,x)\subset\mathcal{A}\left( \frac{\Theta^{s*}(\mu,x)}{\Theta^s_*(\mu,x)} \right) .\]
\end{lemma}
\begin{proof}
	By Proposition \ref{prop:embeddingsandconvergence} and Proposition \ref{prop:Fequivweak*}, we have that for any $(\nu,y)\in \textnormal{Tan}(\mu,x)$ there exists a sequence $r_i\rightarrow 0$, a metric space $(Z,\lambda,z)$ and  isometries $(T_{r_i}(\mu,d,x)),(\nu,\rho,y)\rightarrow (Z,\lambda,z)$ such that
	\begin{equation*}
	 \frac{\mu_i}{\mu(B(x,r_i))}\xrightharpoonup{*}\nu.
	\end{equation*}
	Thus by (\ref{balsam}) we have 
	\begin{align*}
		\nu(B(y,R))\leq \nu(U(y,R+\varepsilon))&\leq\liminf_{i\rightarrow \infty}\frac{\mu(B(x,(R+\varepsilon)r_i))}{\mu(B(x,r_i))}\\
		&=(R+\varepsilon)^s\liminf_{i\rightarrow \infty} \frac{\mu(B(x,(R+\varepsilon)r_i))}{(2(R+\varepsilon)r_i)^s}\cdot \frac{(2r_i)^s}{\mu(B(x,r_i))}\\
		&\leq(R+\varepsilon)^s\frac{\Theta^{s*}(\mu,x)}{\Theta^s_*(\mu,x)}\xrightarrow{\varepsilon\rightarrow 0}R^s\frac{\Theta^{s*}(\mu,x)}{\Theta^s_*(\mu,x)}.
	\end{align*}
	Similarly, by (\ref{fir}),
	\[\nu(B(y,R))\geq \limsup_{i\rightarrow \infty}\frac{\mu(B(x,(R+\varepsilon)r_i))}{\mu(B(x,r_i))}\geq R^s\frac{\Theta_*^s(\mu,x)}{\Theta^{s*}(\mu,x)}\]
	which concludes the first statement. The particular case follows from \ref{lemma:rescalemetrics2}.
\end{proof}
\begin{corollary}\label{remark:ahlfors}
	For an $\mathcal{H}^s$-measurable set $E\subset X$ with $\mathcal{H}^s(E)<\infty$, we can combine Lemma \ref{lemma:ahlforsregular} and Theorem \ref{theorem:density}, to see that for $\mathcal{H}^s$-a.e. $x\in E$,
	\[R^s\Theta_*^s(\mathcal{H}^s \llcorner E,x)\leq\nu(B(y,R))\leq R^s\frac{1}{\Theta^s_*(\mathcal{H}^s \llcorner E,x)}\]
	for all $(\nu,y)\in\textnormal{Tan}(\mathcal{H}^s \llcorner E,x)$.
\end{corollary}

%\begin{prop}\label{prop:decompdensity}
%	Let $(X,d)$ be a complete metric space and $E\subset X$ be $\mathcal{H}^1$-measurable with $\mathcal{H}^1(E)<\infty$. Suppose that for $\mathcal{H}^1$-a.e. $x\in E$, $\Theta_*(E,x)>0$ and
%	\[\textnormal{Tan}(X,d,\mathcal{H}^1\llcorner E,x)\subset \mathcal{C}^*.\]
%	Then there exists compact $C_i\subset E$ such that
%	\begin{enumerate}
%		\item  \[\mathcal{H}^1\left(E\backslash \bigcup_{i\in \mathbb{N}}C_i\right)=0.\]
%		\item for each $i\in\mathbb{N}$ there is a $d_i>0$ such that for all $x\in C_i$
%		\[d_ir\leq \mathcal{H}^1(B(x,r)\cap S)<4r,\qquad\forall r\in(0,d_i).\]
%		\item GTA.
%	\end{enumerate}
%	
%\end{prop}

	\section{Connected Tangents of Metric Spaces}\label{connectedtangents}
	%As yet, we have not assumed any additional structure on our tangents but here we move on to consider connected tangents in particular. Recall that in the geometric argument outlined in the introduction, the circle pairs were disjoint so that the errant measure in the common region of each pair was distinct from any other pair. This allowed us to demonstrate a cumulatively large amount of missing measure and thus produce a contradiction. We have already alluded to the fact that we can expand the empty common region of a circle pair to an empty channel running throughout the set, however we will still need to be able to demonstrate that any errant measure found in a particular channel is going to be distinct from the measure found in any other channel. To do this, we require more information about the behaviour of connected tangents. In Lemma \ref{lemma:uniformquasi} we are able to quantify in some sense how far the tangents can wander which will later allow us to find a scale at which we may impose a notion of disjointedness. 

In this section we show how the hypothesis that all tangent spaces are connected improves to a stronger, quantitative notion of connectedness.

Let $\mathcal{C}$ be the collection of connected separable metric spaces and let $\mathcal{C}^*=\{(\mu,d,x)\in\mathbb{M}_*:\supp\mu\in\mathcal{C}\}$ equipped with $d_*$. 

%\begin{definition}[Quasi-paths]

%	Let $\delta>0$. We say that a metric space $(X,d)$ admits a $\delta$-quasi-path between $a,\,b\in X$ if there exists a finite collection of points $\{x_n\}_{n=0}^N\subset X$ such that $d(x_n,\,x_{n+1})<\delta d(a,b)$ with $x_0=a$, $x_N=b$. 

%	 

%	If there is a $\delta$-quasi-path for all pairs of points, then we call $X$ $\delta$-quasi-path-connected. We collect all such spaces into the set $\mathcal{Q}(\delta)$.

%	

%	If there exists an $R>0$ such that for all $a,b\in X$, $B(a,Rd(a,b))$ admits a $\delta$-quasi-path between $a$ and $b$, then we say that $X$ is $R$-locally $\delta$-quasi-path-connected. We collect all such spaces into the set $\mathcal{Q}_{R}(\delta)$.

%\end{definition}

\begin{definition}[Quasi-path-connected]

	Let $\delta>0$. We say that a metric space $(X,d)$ admits a $\delta$-quasi-path between $a,\,b\in X$ if there exists a finite collection of points $\{x_n\}_{n=1}^N\subset X$ such that $x_1=a$, $x_N= b$ and $d(x_n,\,x_{n+1})\leq\delta d(a,b)$ for each $n$. We have the following refinements:

	\begin{enumerate}

		\item $\mathcal{Q}(\delta,R)$ is the collection of metric spaces $(X,d)$ such that for all $a,b\in X$, $B(a,Rd(a,b))$ admits a $\delta$-quasi-path between $a$ and $b$. We say that such an $X$ is $R$-locally $\delta$-quasi-path-connected.

		\item $\mathcal{Q}_p(\delta,R)$ is the collection of pointed metric spaces $(X,d,x)$ such that for all $y\in X$, $B(x,Rd(x,y))$ admits a $\delta$-quasi-path between $x$ and $y$.

		\item $\mathcal{Q}_*(\delta,R)=\{(\mu,y)\in\mathbb{M}_*:(\supp\mu,y)\in \mathcal{Q}_p(\delta,R)\}.$

		\item $\mathcal{Q}_*(\delta, R, C)=\mathcal{Q}_*(\delta, R)\cap \mathcal{A}(C).$

	\end{enumerate}

\end{definition}

\begin{lemma}\label{remark:decompose}
	For $\delta,R,C>0$, let $(X,d,\mu,a)\notin \mathcal{Q}_*(\delta,\,R,\,C)$. Then there exists a point $b\in \supp\mu$ such that there exists a decomposition of $\supp\mu \cap B(a,Rd(a,b))$ into sets $A$ and $B$ containing $a$ and $b$ respectively with 	
	\[\textnormal{dist}(A,B)>\delta d(a,b).\]
\end{lemma}
\begin{proof}
	Suppose $(X,d,\mu,a)\notin \mathcal{Q}_*(\delta,\,R)$. Then there exists a point $b\in \supp\mu$ such that there is no $\delta$-quasi-path between $a$ and $b$ in $B(a,R d(a,b))$.
	Define
	\begin{multline*}
	A=\{x\in \supp\mu\cap B(a,Rd(a,b)):\\ \textnormal{there exists a $\delta$-quasi-path between $a$ and $x$ in $B(a,Rd(a,b))$}\}
	\end{multline*}
	and $B=\supp\mu\cap B(a,Rd(a,b))\setminus A$.
	Then $a\in A$ and $b\in B$.

	Suppose that $\textnormal{dist}(A,B) \leq \delta (a,b)$.
	Then, since $\mu\in\mathcal{A}(C)$, $A$ and $B$ are compact.
	Thus there exists $a'\in A$ and $b'\in B$ such that $d(a',b')\leq \delta d(a,b)$.
	Then there exists a $\delta$-quasi-path between $a$ and $b$ in $B(a,Rd(a,b))$, a contradiction.
\end{proof}
\begin{lemma}
	Suppose $E\subset X$ is such that there exists a $\delta$-quasi-path connecting $x,y\in E$ in $B(x,Rd(x,y))$ for some $R>1$. Suppose also that $F\subset X$ satisfies
	\begin{equation} \label{sleepy}
		d_{H(x)}(E,F)<\varepsilon d(x,y),\qquad \text{where }\varepsilon<\min\left\{\frac{1}{(R+1)d(x,y)^2},1\right\}.
	\end{equation}
	Then $F\cup\{x,y\}$ contains a $(\delta+2\varepsilon)$-quasi-path connecting $x$ and $y$ in \[B(x,(R+\varepsilon)d(x,y)).\]
\end{lemma}
\begin{proof}
	Let $\{x_n\}\in E$ be the $\delta$-quasi-path connecting $x,y\in E$. Note that
	\[\frac{1}{\varepsilon d(x,y)}>(R+\varepsilon)d(x,y).\]
	Thus by (\ref{sleepy}), each $x_n$ is contained in $B(x,1/\varepsilon d(x,y))$ and by (\ref{sleepy}) there exists an $\hat{x}_n\in F$ such that
	\[d(x_n,\hat x_n)<\varepsilon d(x,y).\] 
	Thus, the triangle inequality gives us
	\[d(\hat x_n,\hat x_{n+1})\leq d(\hat x_n,x_n)+d(x_n,x_{n+1})+d(\hat x_{n+1},x_{n+1})\leq (\delta+2\varepsilon)d(x,y). \]
	Further, by assumption we have
	\[d(x,x_n)\leq Rd(x,y)\qquad \forall n\in\mathbb{N}\]
	thus another application of the triangle inequality tells us that $\{\hat x_n\}\cup \{x,y\}$ is the required $(\delta+2\varepsilon)$-quasi-path in $(F\cup \{x,y\})\cap B(x,(R+\varepsilon)d(x,y)) $. 
\end{proof}
%\begin{lemma}\label{quasiclose}\unsure{arg1: Proposition 2.4, Lemma 2.7 and step 3 of Lemma 4.1 all use the same idea: Being GH close to something in $Q(\delta,R,C)$ implies you are in $Q(\delta',R',C)$.}
%	Suppose $E,F\subset X$ satisfy $E\in Q(\delta,R,C)$ and $F\in \mathcal{A}(C)$. Suppose also that
%	\[d_{GH}(E,F)<\varepsilon.\]
%	Then $F\in Q(\delta+2\varepsilon,R+2\varepsilon,C)$.
%\end{lemma}
%\begin{proof}
%	content...
%\end{proof}

\begin{lemma}\label{close?}
	Let $(X,\mu,x)$ be a pointed metric space and $\varepsilon,\delta\in(0,1)$. Let $\hat x\in B(x,(1/\varepsilon - \delta))$ be such that
	\begin{equation}\label{pawpaw}
		\mu(B(\hat x,\delta))>\varepsilon.
	\end{equation} 
	If $K_\mu\subset X$ is such that 
	\begin{equation}\label{teatime}
		\mu\left(B\left(x,\frac{1}{\varepsilon}\right)\backslash K_\mu\right)<\varepsilon,
	\end{equation}
	then $K_\mu$ must contain a point $k_\mu\in K_\mu$ such that
	\[d(\hat x,k_\mu)<\delta.\]
\end{lemma}
\begin{proof}
	Because $B(\hat{x},\delta)\subset B(x,1/\varepsilon) $, combine (\ref{pawpaw}) with (\ref{teatime}) to see that there must be some point $k_\mu\in K_\mu\cap B(\hat{x},\delta)$. 
\end{proof}

\begin{corollary}\label{pathypath}
	For $R>1$, let $(X,\mu,x)$ be a metric measure space, $y\in \supp\mu$ and suppose that $\supp \mu$ contains a $\delta$-quasi-path $\{x_n\}$ in $B(x,Rd(x,y))$ connecting $x$ and $y$. Suppose also that for $\varepsilon,\delta \in (0,1)$,
	\begin{equation}\label{pips}
		\mu(B(x_n,\delta d(x,y)))>\varepsilon\qquad\forall\,n\in\mathbb{N}.
	\end{equation}
	Then for any $K_\mu\subset X$ containing $x,\,y$ and satisfying
	\begin{equation}\label{squeaky}
		\mu(B(y,(R+1)d(x,y))\backslash K_\mu)<\varepsilon,
	\end{equation}
	$K_\mu\cap B(y,(R+\delta)d(x,y))$ contains a $3\delta$-quasi-path connecting $y$ and $t$.
\end{corollary}
\begin{proof}
	Lemma \ref{close?} implies that, for each $n\in\mathbb{N}$, there exists some point $q_n\in K_\mu \cap B(x_n,\delta d(x,y))$. Hence,
	\[d(q_n,q_{n+1})\leq d(q_n,x_n)+d(x_n,x_{n+1})+d(q_{n+1},x_{n+1})\leq 3\delta d(x,y).\]
	Thus the collection $\{y, t, q_n\}$ forms the required $3\delta$-quasi-path in $K_\nu\cap B(y,(R+\delta)d(y,t))$.
\end{proof}

%\begin{lemma}
%	Let $(\mu,x),\,(\nu,y)\in\mathbb{M}_*$ where $(\nu,y)\in\mathcal{Q}_*(\delta,\,R,\,C)$ and let $\hat{x}\in \supp\mu$. If 
%	\begin{equation}\label{cow}
%		\mu(B(\hat x,\delta d(x,\hat x)))\geq \varepsilon\qquad\text{where } \varepsilon\leq \min \left\{\frac{\delta d(x,\hat x)}{4C},\frac{1}{(1+\delta)d(x,\hat x)}\right\}.
%	\end{equation}
%	Then
%	\begin{equation}
%		d_*((\mu,x),\,(\nu,y))<\varepsilon
%	\end{equation}
%	implies that there exists a $\delta+$-quasi-path between $x$ and $\hat x$ in $\supp\mu\cap B(x,(R+\delta)d(x,\hat x))$.
%\end{lemma}
%
%\begin{proof}
% 	By ........., there exists compact
% 	$x\in K_\mu\subset\supp\mu\cap B(x,1/\varepsilon)$ and $y\in K_\nu \subset \supp\nu\cap B(y,1/\varepsilon)$
% 	such that 
% 	\begin{enumerate}[left=1cm, label=\textbf{M\arabic*.},ref=\textbf{M\arabic*.}]
% 		\item $\max\{\mu(B(x,1/\varepsilon)\backslash K_\mu),\nu(B(y,1/\varepsilon)\backslash K_\nu)\}<\varepsilon$;
% 		\item $d_{pmGH}((K_\mu,\mu\llcorner K_\mu,x),(K_\nu,\nu\llcorner K_\nu,y))<3\varepsilon$.
% 	\end{enumerate}	
% 	Let $(K_\mu,x)(K_\nu,y)\to (Z,z)$ be an embedding as given by Proposition? that gives us
% 	\[d_{H(z)}(K_\mu,K_\nu)<3\varepsilon.\]
%	By Proposition \ref{close?}, $K_\mu$ must contain a point $k_\mu\in K_\mu$ such that
%	\[d(\hat x,k_\mu)<\delta d(x,\hat{x}).\]
%	Then, by \ref{largesubsets2}, there must be a point $\hat y\in K_\nu$ such 
%	\[d_Z(k_\mu,\hat{y})<3\varepsilon.\]
%	Because $(\nu,y)\in\mathcal{Q}_*(\delta,\,R,\,C)$, there exists a $\delta$-quasi-path between $y$ and $\hat y$.
%\end{proof}
%
%\newpage 
In general, $\mathcal{Q}(\delta,R)\subsetneq \mathcal{Q}_p(\delta,R)$. Indeed, if $X=\{0\}\cup\{1/2^n\}_{n\in\mathbb{N}}$ with the Euclidean metric, then $(X,\|\cdot\|_2,0)\in \mathcal{Q}_p(1,1/2)$ but for any $n$, there is no 1/2-quasi-path from $1/2^n$ to $1/2^{n+1}$ so $X\notin \mathcal{Q}(1,1/2)$. However, in the context of tangents, \ref{lemma:rescalemetrics2} shows that $\mathcal{Q}_p(\delta,R)$ and $\mathcal{Q}(\delta,R)$ are essentially equivalent.
\begin{corollary}
	Let $(X,d,\mu)$ be an asymptotically doubling metric measure space. Suppose that for $\mu$-a.e. $x\in\supp\mu$ there exists $R_x,\delta_x >0$ such that
	\[\textnormal{Tan}(\mu,x)\subset \mathcal{Q}_*(\delta_x,R_x).\]
	Then for $\mu$-a.e. $x\in\supp\mu$, 
	\[ (\nu,y)\in \textnormal{Tan}(\mu,x) \quad \implies\quad  (\supp\nu,y) \in \mathcal{Q}(\delta_x,R_x). \] 
\end{corollary}
We now show that if all tangents of a space are connected, then all tangents are in fact quasi-path-connected.
\begin{prop}\label{lemma:uniformquasi}
	Let $(X,\,d)$ be a complete metric space and $\mu\in\mathcal{M}_{loc}(X)$. Suppose that for $\mu$-a.e. $x\in\supp\mu$ we have
	\begin{enumerate}[left=1cm, label=\textbf{(\roman{enumi})}, ref= \textbf{(\roman{enumi})} ]
		\item $0<\Theta^1_*(\mu,x),\,\Theta^{1*}(\mu,x)<\infty$; \label{uniformquasi1}
		\item $\textnormal{Tan}(\mu,x)\subset\mathcal{C}^*.$ \label{uniformquasi2}
	\end{enumerate} 
	Then for $\mu$-a.e. $x\in \supp\mu$ and any $\delta>0$ there exists $R_x,C_x>0$ such that $\textnormal{Tan}(\mu,x)\subset \mathcal{Q}_*(\delta,R_x,C_x)$. 
\end{prop}

\begin{proof}
	Fix $\delta>0$ and an $x$ that satisfies \ref{uniformquasi1}, \ref{uniformquasi2}, has a compact tangent space as given by Proposition \ref{prop:tancompact} and such that we may apply \ref{lemma:rescalemetrics1} and \textbf{(2)}. Note that $\mu$-a.e. $x\in X$ satisfies theses requirements. First, recall that Lemma \ref{lemma:ahlforsregular} gives us a $C_x>1$ such that
	\begin{equation}\label{eq:ahlf}
		\frac{1}{C_x}r\leq \nu(B(z,r))\leq C_x r \qquad \forall\,z\in\supp\nu,\,r>0.
	\end{equation} 
	To demonstrate the existence of $R_x$ we work by contradiction, see Figure 1 for a geometric description of the proof. If there exists no such $R_x>0$, then for any $n\in\mathbb{N}$ there exists a tangent measure $(\nu_n,\rho_n,a_n)\in\textnormal{Tan}(\mu,x)$ such that $(\nu_n,\rho_n,a_n)\notin \mathcal{Q}_*(\delta,n,C_x)$. Thus, by Lemma \ref{remark:decompose}, there exists a point $b_n\in \supp\nu_n$ such that $\supp\nu_n\cap B(a_n,n\rho_n(a_n,b_n))$ can be decomposed into sets $A_n$ and $B_n$ containing $a_n$ and $b_n$ respectively with
	\[\textnormal{dist}_{\rho_n}(A_n,B_n)>\delta \rho_n(a_n,b_n).\] 
	We use the $(\nu_n,a_n)$ to construct another sequence of elements of $\textnormal{Tan}(\mu,x)$. By \ref{lemma:rescalemetrics1}, we may rescale the metrics of each tangent and suppose that $\rho_n(a_n,b_n)=1$ (we omit the subsequent scaling of the measure by $C$ for ease of notation). By compactness of $\textnormal{Tan}(\mu,x)$, there exists a subsequence $(\nu_{n_k},\rho_{n_k},a_{n_k})$ such that, for any $\varepsilon>0$, there exists a $K>0$ such that
	\[d_*((\nu_{n_k},a_{n_k}),(\nu_{n_l},a_{n_l}))<\varepsilon\qquad \forall\,k,l>K.\] 
	By Lemma \ref{prop:embeddingsandconvergence}, we may consider $(\nu_{n_k},a_{n_k})$ to be in a common ambient space $(Z,d_Z)$ with each $a_{n_k}$ mapped to a common point $a\in Z$. For brevity we relabel the subsequence to be indexed once again by $n$.
    \par
	Finally, by \ref{lemma:rescalemetrics2} we observe that $(\nu_n,b_n)$ are also tangent measures. Hence there exists a sequence $\{(\nu_n, \rho_n,b_n)\}_{n\in\mathbb{N}}\subset\textnormal{Tan}(\mu,x)$ such that $\supp\nu_n\cap B(a,n)$ has decomposition $A_n\cup  B_n$ satisfying
	\[\textnormal{dist}_{Z}(A_n,B_n)>\delta.\]
	By compactness of $\textnormal{Tan}(\mu,x)$, $\{(\nu_{n},b_{n})\}$ converges in $d_*$ (up to subsequence) to a limit $(\nu,b)\in\textnormal{Tan}(\mu,x)$ with $a\in\supp\nu$. Since $\rho_n(a_n,b_n)=1$ for each $n$, $d_Z(a,b)=1$.
    \par 
    By \ref{uniformquasi2} we have that $(\nu,b)\in\mathcal{C}^*$, thus there must exist a finite $\delta/3$-quasi-path $\{z^i\}\subset\supp\nu$ between $a$ and $b$ and a radius $R$ such that $\{z^i\}\subset B(b,R)$. 
    \par 
	By Proposition \ref{prop:properties_of_weak_con}, for every $z_i$, there must exist a sequence $z^i_n\in\supp\nu_n$ such that $z^i_n\xlongrightarrow{n\rightarrow\infty} z^i$. Since our quasi-path is finite we may pick an $N>R+\delta$ sufficiently large such that $d_Z(z^i_N,z^i)<\delta/3$ for all $i$. However, this implies that we have a path $z^i_N\in\supp\nu_N$ such that
	\[d_Z(z^i_N,z^{i+1}_N)\leq d_Z(z^i_N,z^i)+d_Z(z^i,z^{i+1})+d_Z(z^{i+1},z^{i+1}_N)\leq \delta. \]
	By another application of the triangle inequality we also have that $\{z^i_N\}\subset B(b,R+\delta)\subset B(b,N)$. This contradicts our initial assumption that no such $\delta$-path exists in $\supp\nu_N\cap B(b,N)$. 

%	For all $n>R$ sufficiently large, we have that $d_*((\nu,b),(\nu_n,b_n))<1/R$. Fix one such $n$ and take compact $K_{\nu_n}$ and $K_\nu$, as in Theorem \ref{theorem:largesubsets}, such that

%	\begin{equation}\label{eq:largesubsets}

%		\max\{\nu(B(b,R)\backslash K_\nu),\nu_n(B(b_n,R)\backslash K_{\nu_n})\}<1/R

%	\end{equation}

%	and

%	\begin{equation}\label{eq:getclose}

%		d_{pmGH}((K_\nu,\nu\llcorner K_\nu,b),(K_{\nu_n},\nu_n\llcorner K_{\nu_n},b_n))<3/R.

%	\end{equation}

%	Because $\{z_i\}\subset B(b,R)$, we have that there exists $z_i$ such that $\textnormal{dist}_Z(z_i,A_n)\geq\delta/4$ and  $\textnormal{dist}_Z(z_i,B_n)\geq\delta/4$. Thus we have that $B(z_i,\delta/8)\cap A_n=B(z_i,\delta/8)\cap B_n=\emptyset$ and using (\ref{eq:ahlf}) we see that

%	\[\nu(B(z_i,\delta/8))\geq\frac{\delta}{8C}>1/R.\]

%	Combining this with (\ref{eq:largesubsets}) means there exists $z\in K_\nu\cap B(z_i,\delta/8)$ with $\textnormal{dist}_Z(z,A_n),\textnormal{dist}_Z(z,B_n)\geq\delta/8$. But this means 

%	\[\textnormal{dist}_Z(z,K_{\nu_n})\geq \delta/8>3/R\]

%	which contradicts (\ref{eq:getclose}), thus concluding the proof.

\end{proof}

\begin{figure}[t]
	\centering

	\def\svgwidth{1\textwidth}

	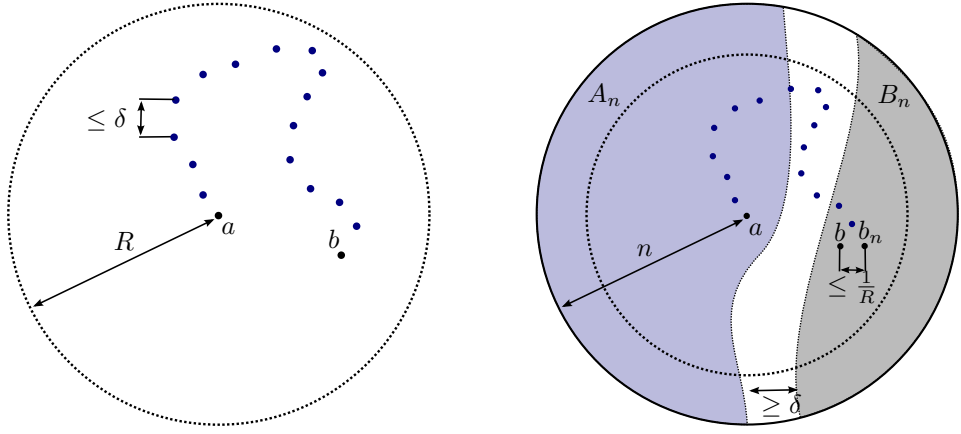

	\caption{The contradiction in the proof of Lemma \ref{lemma:uniformquasi}. On the left, the limit of the sequence $(\nu_n,b_n)$ must have a radius which contains a $\delta$-quasi-path. On the right, for sufficiently large $n$, the existence of the $\delta$-quasi-path implies there must be points of $\supp\nu_n$ in the gap.}

\end{figure}
We now decompose $X$ into countably many Borel pieces on which we have a uniform bound on the $R_x$ and $C_x$ obtained from Proposition \ref{lemma:uniformquasi}.
\begin{lemma}[Lemma 5.4 \cite{bate}]\label{lemma:borel}
	Let $(X,d,\mu)$ be a metric measure space and suppose that $\mathcal{S}\subset\mathbb{M}_*$. Then for any $k,\varepsilon>0$,
	\[C_{k,\varepsilon}:=\{x\in \supp\mu:d_*(T_r(\mu,d,x),\mathcal{S})\leq\varepsilon,\,\forall\,r<k\}\]
	is a closed subset of $X$. In particular, if $E\subset X$ is a Borel set of points satisfying Proposition \ref{prop:tancompact} and $\mathcal{S}\subset\mathbb{M}_*$ is closed, then
	\[\{x\in S:\textnormal{Tan}(\mu,d,x)\subset\mathcal{S}\}\]
	is Borel.
\end{lemma}

\begin{lemma}\label{lemma:closed}
	For any $\delta<1/2$ and $R,C\geq 1$, $\mathcal{Q}_*(\delta,R,C)$ is a closed subset of $(\mathbb{M}_*,d_*)$.
\end{lemma}
\begin{proof}
	Take a sequence $(\mu_i,x_i)\in \mathcal{Q}_*(\delta,R,C)$ that converges to $(\mu,x)\in \mathbb{M}_*$. By Proposition \ref{prop:embeddingsandconvergence} we may consider the sequence and limit to be contained in complete and separable metric space $(X,d)$ with each $x_i$ mapped to $x$. Then by Proposition \ref{prop:Fequivweak*} we have $\mu_i\xrightharpoonup{*}\mu$.
    \par 
	Let $y\in\supp\mu$. We will demonstrate that there exists a $\delta$-quasi-path between $x$ and $y$. First, by Proposition \ref{prop:properties_of_weak_con_ahlfors}, we know that $\mu\in\mathcal{A}(C)$. By Proposition \ref{prop:properties_of_weak_con}, $y\in \supp\mu$ implies that there exists a sequence $y_i\in\supp\mu_i$ such that $y_i\rightarrow y$. Since for each $i$ we have that $\mu_i\in \mathcal{Q}_*(\delta,R,C)$, we also have a sequence of $\delta$-quasi-paths $\{y_i^n\}_{n\in\mathbb{N}}\subset \supp\mu_i\cap B(x,Rd(x,y_i))$. 
    \par 
	Let us assume that $\mu\notin \mathcal{Q}_*(\delta,R,C)$, then, as in Lemma \ref{remark:decompose}, there exists a point $y\in \supp\mu$ and two open sets $A_1,A_2\subset B(x,Rd(x,y))$ such that
	\begin{enumerate}[left=1cm, label=\textbf{(\arabic*)},ref=\textbf{(\arabic*)}]
		\item $\supp\mu\cap B(x,Rd(x,y))\subset A_1\cup A_2$; \label{spiderplant}
		\item $x\in A_1$ and $y\in A_2$;
		\item $\textnormal{dist}(A_1,A_2)>\delta d(x,y)$.
	\end{enumerate}
	Then $B= B(x,Rd(x,y))\backslash A_1\cup A_2$ is disjoint from $\supp\mu$. Now, because each $\mu_i\in \mathcal{Q}_*(\delta,R,C)$, no such similar decomposition of $B(x,Rd(x,y))$ is possible, i.e. for each $\mu_i$ there exists some $z_i\in B\cap \supp\mu_i$. Hence, by using the Ahlfors regularity of $\mu_i$ we see that we must have
	\[\mu_i(B)>0.\]
	Thus
	\begin{equation*}
		\mu(B)\geq \limsup_{i\rightarrow \infty}\mu_i(B)>0
	\end{equation*}
	which directly contradicts item \ref{spiderplant}, concluding the proof.
\end{proof}

\begin{corollary}\label{decomposemebaby}
	Let $(X,d)$ be a complete metric space and $E\subset X$ be $\mathcal{H}^1$-measurable with $\mathcal{H}^1(E)<\infty$. Assume that for $\mathcal{H}^1$-a.e. $x\in E$,
	\[\Theta_*(E,x)>0,\qquad\text{and}\qquad\textnormal{Tan}(X,d,\mathcal{H}^1\llcorner E,x)\subset \mathcal{C}^*.\]
	Then for any $\delta>0$, there exists a compact set $F\subset E$ with $\mathcal{H}^1(F)>0$ satisfying the following:
	\begin{enumerate}[left=1cm, label=\textbf{S\arabic*.},ref=\textbf{S\arabic*.}]
		\item There exists $\hat{d}>0$ such that for $\mathcal{H}^1$-a.e. $x\in F$,
		\begin{equation*}
			\hat{d}\leq\Theta_*(F,x),\qquad \Theta^*(F,x)\leq 1.
		\end{equation*}\label{me}
		\item There exists $R>0$ such that for $\mathcal{H}^1$-a.e. $x\in F$, \[\textnormal{Tan}(\mathcal{H}^1\llcorner F,x)\subset \mathcal{Q}_*(\delta,R,1/\hat{d})\cap \mathcal{C}^*.\]\label{you}
	\end{enumerate} 
\end{corollary}
\begin{proof}
	Fix $\delta>0$.  By Theorem \ref{theorem:density}, we can find a set $E'\subset E$ of positive measure with constants $\hat d,\,r_0>0$ such that for all $x\in E'$,
	\begin{equation*}
		2\hat d r\leq \mathcal{H}^1(E\cap B(x,r))\leq 4r\qquad \forall \,r<r_0.
	\end{equation*}
	For any $R>0$, define the set
	\[F_R=\{x\in E: \textnormal{Tan}(\mathcal{H}^1\llcorner E,x)\subset \mathcal{Q}_{*}(\delta,R,1/\hat{d}) \}.\] 
	By Proposition \ref{lemma:uniformquasi} and Proposition \ref{lemma:ahlforsregular}, we know that $F_R$ monotonically increase with $R$ to cover $\mathcal{H}^1$-almost all of $E'$. By Lemma \ref{lemma:closed} and Lemma \ref{lemma:borel}, $F_R$ are measurable thus there exists $R>0$ such that
	\[\mathcal{H}^1(E'\cap F_R)>0.\]
	Because $E'\cap F_R$ is $\mathcal{H}^1$-measurable, we may find some compact subset $F\subset E'\cap F_R$ with $\mathcal{H}^1(F)>0$. Using Corollary \ref{cor:densityofrestrictedsets} we have that for $\mathcal{H}^1$-a.e. $x\in F$
	\begin{equation*}
		\hat{d}\leq\Theta_*(F,x),\qquad \Theta^*(F,x)\leq 1.
	\end{equation*}
	So $F$ satisfies \ref{me}. Recalling Remark \ref{remark:densedouble}, $\mathcal{H}^1\llcorner E$ is asymptotically doubling, thus by Lemma \ref{cor:restrict}, for $\mathcal{H}^1$-a.e. $x\in F$,
	\begin{equation*}
		\textnormal{Tan}(\mathcal{H}^1\llcorner F,x)=\textnormal{Tan}(\mathcal{H}^1\llcorner E,x)\subset \mathcal{Q}_*(\delta,R,1/\hat{d})\cap \mathcal{C}^*,
	\end{equation*}
	which gives us \ref{you} as required.
\end{proof}

	\section{Besicovitch Partitions}\label{besicovitchpartitions}
	In this section we show that a large proportion of a purely 1-unrectifiable set in a metric space can be covered by \emph{Besicovitch partitions}.
A Besicovitch partition is an adaptation of Besicovitch's ideas on circle pairs in order to work with quasi-connected tangents.
Indeed, in the next section we will see that a set with a suitable collection of Besicovitch partitions cannot have quasi-path connected tangents.

In place of the Besicovitch covering theorem, we use the Ahlfors type estimate \eqref{sportsdirect} to bound the size of the cover.
\begin{notation}
	For a collection $\mathcal{B}$ and a set $B\in\mathcal{B}$, let 
	\[B_\cap:=\{B'\in \mathcal{B}:B\cap B'\neq \emptyset\}.\]
\end{notation}
\begin{lemma}\label{lemma:uniformcovering}
	Let $E\subset X$ be a totally bounded set such that for some $r_0,C>0$ we have
	\begin{equation}\label{sportsdirect}
		\frac{1}{C}r^n\leq \mathcal{H}^n(E\cap B(x,r))\leq Cr^n\qquad \forall r<r_0,\,x\in E.
	\end{equation}
	Then for any $r<r_0$, there exists a finite cover $\mathcal{B}$ of $E$ by balls such that
	\begin{enumerate}[left=1cm, label=\textbf{T\arabic*.},ref=\textbf{T\arabic*.}]
		\item \label{uniformrad} the radius $r$ of every ball is the same;
		\item \label{closecover} $\sum_{B\in\mathcal{B}}\textnormal{diam}(B)\leq 8^nC\mathcal{H}^n(E)$;
		\item \label{martin} for any $B\in \mathcal{B}$, $\# B_\cap\leq C^2 12^n$.
	\end{enumerate}
\end{lemma}	
\begin{proof}
	Fix $r<r_0$ and let $\mathcal{C}=\{B(x_i,r/2):x_i\in E\}_{i=1}^N$ be a finite covering of $E$, the existence of which is assured by the assumption of compactness. We first modify this covering such that the $x_i$ are distance at least $r/2$ apart. Iteratively pick balls such that the centre of one is not included in another and collect into a new collection $\mathcal{C}'$ until no this is no longer possible. This terminates after finitely many steps. Set $\mathcal{B}=2\mathcal{C}'$. We have that $\mathcal{B}$ covers $\mathcal{C}$ and is made of balls of radius $r$ with $r/2$-separated centres.
	
	%Then $2\mathcal{B}'$ covers $\mathcal{B}$: if $B(x_i,r/2)\in \mathcal{B}\cap (2mathcal{B}')^c$ then $x_i\in B(x_j,r/2)$ for some $B(x_j,r/2)\in mathcal{B}'$ else $B(x_i,r/2)$ satisfies the disjointedness requirement and can thus be added to $mathcal{B}'$. 
	
	Now pick a ball $B\in \mathcal{B}$ and note that $\bigcup_{A\in B_\cap} A \subset 3B$ and that the balls formed taking the same centres but a quarter the radius, $B_\cap/4$, are disjoint. Thus, using (\ref{sportsdirect}),
	\begin{align*}
		\frac{1}{C}(r/4)^n(\# B_\cap)&\leq \mathcal{H}^n(E\cap \bigcup_{A\in B_\cap}A/4)\\
		&\leq \mathcal{H}^n(E \cap 3B)\leq C(3r)^n
	\end{align*}
	which gives \ref{martin}.

	To prove \ref{closecover}, note that, since the elements of $\mathcal B/4$ are disjoint,
	\[
	\sum_{B_i\in \mathcal{B}}\mathcal{H}^n(E\cap B(x_i,r/4))
	=
	\mathcal{H}^n(E\cap \bigcup_{B\in \mathcal{B}} B/4)
	\leq
	\mathcal{H}^n(E).\]
	Also, by (\ref{sportsdirect}),
	\[
	(2r)^n = 8^n C \left(\frac{1}{C}(r/4)^n\right)
	\leq
	8^n C \mathcal{H}^n(E\cap B(x_i,r/4)).\]
	Combining these two inequalities gives \ref{closecover}.
\end{proof}

	We now consider continua in metric space.
	Recall that a continuum $L\subset X$ is a compact connected set with $\mathcal{H}^1(L)<\infty$. Recall also that a set with an underlined subscript denotes a closed neighbourhood that set:
\[L_{\boldsymbol{\underline\varepsilon}}:=\{x\in X:\textnormal{dist}(x,L)\leq \varepsilon\}.\]
\begin{lemma}\label{lemma:urgh}
	Let $F\subset X$ be a compact set, $l>0$ and $L^i$ be a sequence of continua in $X$ with $\mathcal{H}^1(L^i)<l$ such that  $\textnormal{dist}(F,x)\leq 1/i$ for all $x\in L^i$. Then there exists a subsequence and a continuum $L$ in $F$ such that 
	\begin{enumerate}
		\item $\mathcal{H}^1(L)\leq l$;
		\item for all $m\in\mathbb{N}$, there exists an $I>0$ such that (up to subsequence)
		\[L_{\boldsymbol{\underline{\nicefrac{1}{m}}}}\supset L_{\boldsymbol{\underline{\nicefrac{1}{i}}}}^{i}\qquad\forall\,i>I.\]
	\end{enumerate}
\end{lemma}
\begin{proof}
	By \cite[Theorems 4.2.1 \& 4.4.8]{luigi}, the $L^i$ may be covered by 1-Lipschitz images of $[0,2l]$.
	By \cite[Theorem 4.4.3]{luigi}, there exists a subsequence of these maps converging uniformly to some $\gamma\colon [0,2l]\to F$.
	Setting $L=\gamma([0,2l])$ gives the desired result.
\end{proof}

The following lemma is a variant of \cite[Lemma 8]{paper:besi2}.
In place of an ambient Euclidean space, we will use the Kuratowski embedding to isometrically embed a separable metric space $X$ into $\ell_\infty$, the space of real bounded sequences with the supremum norm.
In order to use compactness arguments, we note that, if $F\subset \ell_\infty$ is compact, then so is $CC(F)$, the closed convex hull of $F$.
\begin{lemma}\label{lemma:smallballs}%Besi lemma 8 
	Let $X$ be a metric space and let $F\subset X$ be compact, $\mathcal{H}^1$-measurable and purely 1-unrectifiable with $\mathcal{H}^1(F)<\infty$. For $ \varepsilon,\,l>0 $, there exists a radius $ r>0 $ such that for any continuum $L$ in $X$ with $\mathcal{H}^1(L)<l$,
	\[ \mathcal{H}^1(F\cap L_{\boldsymbol{\underline{r}}})<\varepsilon. \]
\end{lemma}
\begin{proof}
	Identify $X$ with its isometric image in $\ell_\infty$ and suppose that the hypothesis does not hold for some $\varepsilon>0$. Then for all $i\in\mathbb{N}$, there exists a sequence of continua $L^i\subset X$ such that
	\[\mathcal{H}^1(L^i)<l\]
	but 
	\[\mathcal{H}^1(F\cap L_{\boldsymbol{\underline{\nicefrac{1}{i}}}}^{i})\geq \varepsilon.\]
	We may additionally assume that $\textnormal{dist}(CC(F),x)\leq 1/i$ for all $x\in L^i$. Indeed,
	because $CC(F)$ is convex, we may replace any part of $L^i$ that lies outside by the line segment connecting the end points of this part, creating a shorter curve.
	Thus we may apply Lemma \ref{lemma:urgh} to find that there exists a continuum $L\subset CC(F)$, such that $\mathcal{H}^1(L)\leq l$ and such that (for a relabelled subsequence of $L^i$ and) all $m\in\mathbb{N}$, there exists an $I>0$ such that
	\[L_{\boldsymbol{\underline{\nicefrac{1}{m}}}}\supset L_{\boldsymbol{\underline{\nicefrac{1}{i}}}}^{i}\qquad\forall\,i>I.\]
	Thus we have
	\[\mathcal{H}^1(F\cap L_{\boldsymbol{\underline{\nicefrac{1}{m}}}} )\geq \mathcal{H}^1(F\cap L_{\boldsymbol{\underline{\nicefrac{1}{i}}}}^{i})\geq\varepsilon. \]
	Then observing the fact that $L=\bigcap_{m}L_{\boldsymbol{\underline{\nicefrac{1}{m}}}}$, and that $\mathcal{H}^1(L)\leq l$, we may use the continuity of $\mathcal{H}^1$ from above to see that
	\begin{equation}\label{pencilcase}
		\mathcal{H}^1(F\cap L)=\mathcal{H}^1(F\cap \bigcap_{m}L_{\boldsymbol{\underline{\nicefrac{1}{m}}}})=\inf \mathcal{H}^1(F\cap L_{\boldsymbol{\underline{\nicefrac{1}{m}}}})\geq\varepsilon.
	\end{equation}
	However, $L$ is 1-rectifiable, thereby (\ref{pencilcase}) contradicts the fact that $F$ is purely $1$-unrectifiable.
\end{proof}
We replace the notion of Besicovitch circle pairs with the following stronger analogue:
\begin{definition}[Besicovitch partition]\label{def:besipart}
	A Besicovitch partition of a set $E\subset X$ is a collection $P=(E_1,E_2,\pardist,p_1,p_2)$ containing
	\begin{enumerate}[left=1cm, label=\textbf{(\roman{enumi})}, ref=Definition \ref{def:-besipart} \textbf{(\roman{enumi})} ]
		\item A partition: $E = E_1\cup E_2$;
		\item A distance: $0<\omega=\textnormal{dist}(E_1,E_2)$;
		\item A pair of points: $p_1\in E_1$ and $p_2\in E_2$ such that $d(p_1,p_2)=\pardist$.
	\end{enumerate}
\end{definition}
We require a notion of disjointness of Besicovitch partitions at a scale corresponding to the $\delta$-path radius that we found via Lemma \ref{lemma:uniformquasi}. In this vein, we say that two Besicovitch partitions $P=(E_1,E_2,\pardist,p_1,p_2)$ and $P'=(E'_1,E'_2,\pardist',p'_1,p'_2)$ are \emph{$R$-disjoint} if $B(p_i,R\pardist)\cap B(p'_j,R\pardist')=\emptyset$ for $i,j\in\{1,2\}$.
 
The following result is in the spirit of \cite[Lemma 9]{paper:besi2}.
However, the construction of the continua in the proof differs quite significantly from the original, which uses the fact that, in the plane, the boundary of a ball is a continuum.
\begin{prop}\label{lemma:besicovtichpartitions}
	Let $(X,d)$ be a complete metric space and $ F\subset X $ compact, purely $1$-unrectifiable such that for some $r_0>0$ and $C\geq 1$, we have
	\begin{equation}\label{penguin}
		\frac{1}{C}r\leq \mathcal{H}^1(F\cap B(x,r))\leq Cr\qquad \forall r<r_0,\,x\in F.
	\end{equation} 
	For any $\varepsilon,\delta>0$ and $R\geq 1$, there exists a finite collection $ \mathcal{P} $ of Besicovitch partitions $P_i=(F_{i,1}F_{i,2},\pardist_i,p_{i,1},p_{i,2})$, $1\leq i\leq n$, such that 
	\begin{enumerate}[left=1cm, label=\textbf{P\arabic*.},ref=\textbf{P\arabic*.}]
		\item $P_i$, $P_j$ are $R$-disjoint for each $1\leq i\neq j\leq n$;\label{besi:distinct}
		\item for each $1\leq i \leq n$, $\pardist_i<\delta$; \label{besi:bounded}
		\item the distances are cumulatively large: \[ \sum_{1\leq i\leq n}\pardist_i>\frac{1-\varepsilon}{3(2R+1)}\mathcal{H}^1(F). \] \label{besi:lots}
	\end{enumerate}
\end{prop}

\begin{proof}
	We identify $F$ with its isometric image in $\ell_\infty$.
	Let $\delta_0>0$ be given by Lemma \ref{lemma:hausdorff}, such that, for all $ r\in (0,\,\delta_0) $,
	\begin{equation}
		\label{delta0}
		\mathcal{H}^1(F\cap \bigcup \mathcal{A}_r)<\sum_{A_i\in \mathcal{A}_r}\textnormal{diam}(A_i)+\frac{\varepsilon}{2}\mathcal{H}^1(F),
	\end{equation}
	where $ \mathcal{A}_r $ is any collection of convex sets with diameter at most $ r $.
	 
	Let
	\[\delta_1<\frac{1}{3(2R+1)}\min\{\delta,\delta_0, r_0\}.\]
	Since $F$ is compact, there exists a finite set $ \{x_n\}_{n=1}^N\subset F $, such that $ F\subset\bigcup_{n=1}^NB(x_n,\delta_1)$. Let $L$ be the polygonal line joining these $x_n$, then $\mathcal{H}^1(L)=\sum d(x_n,x_{n+1})$.
	Let $r<\delta_1$ be given by Lemma \ref{lemma:smallballs} such that, for any continuum $\tilde{L}$ with \[\mathcal{H}^1(\tilde{L})<\mathcal{H}^1(L)+200C^3(R+2)\mathcal{H}^1(F),\]
	we have
	\begin{equation}\label{chalky} 
		\mathcal{H}^1(F\cap \tilde{L}_{\textbf{\underline{r}}})<\frac{\varepsilon}{2}\mathcal{H}^1(F).
	\end{equation} 
	Finally, we satisfy the assumptions of Lemma \ref{lemma:uniformcovering} so way may find a cover of $F$ by a finite collection of balls $\mathcal{A}_1$ that satisfies
	\begin{enumerate}
		\item [\ref{uniformrad}] the radius $\hat r<r/2$ of every ball is the same;
		\item [\ref{closecover}] $\sum_{B\in\mathcal{A}_1}\textnormal{diam}(B)\leq 8C\mathcal{H}^1(F)$;
		\item [\ref{martin}] for any $B\in \mathcal{A}_1$, $\# B_\cap\leq 12 C^2$.
	\end{enumerate}
	
	We now iteratively construct collections $\mathcal A_k$ of balls by adding elements to $\mathcal A_1$ until the union of some $\mathcal A_k$ is connected.
	Suppose that $\mathcal A$ is a collection of balls.
	For each $B\in\mathcal A$ define
	\[\operatorname{Cluster}\,(B)=\left\{\hat B\in\mathcal A :B,\hat B \text{ in same connected component of }\left(\bigcup \mathcal A\right)_{\boldsymbol{\underline{\hat r/4}}} \right\},\]
	and
	\[\text{Cluster}\,{\mathcal{A}}:=\left\{\operatorname{Cluster}\,(B):B\in\mathcal A\right\}.\]
	Also define
	\[\text{Base}\,{\mathcal{A}}:=\bigcup\{C\in \text{Cluster}\,\mathcal{A} :x_n\in C\text{ for some }n\}.\]

	We now add balls to $\mathcal A_k$ as follows.
	Set $\mathcal{P}_0=\mathcal{B}_0=\emptyset$ and, for $k\geq 1$, suppose that $\mathcal A_k$ is defined.
	\begin{enumerate}
		\item Let $F_1=F\cap \text{Base}\,{\mathcal{A}_k} $ and $F_2=F\backslash F_1$.
		\item Find points $p_1\in F_1$ and $p_2\in F_2$ such that $\pardist:=d(p_1,p_2)=\textnormal{dist}(F_1,F_2)$, so that $(F_1,F_2,\pardist,p_1,p_2)$ is a Besicovitch partition that satisfies \ref{besi:bounded}.
		\item Let $b$ be the midpoint of $p_1$ and $p_2$ and let $B=B(b,\frac{2R+1}{2}\pardist)$.
		\item Set $\mathcal{P}_k=\mathcal{P}_{k-1}\cup\{(F_1,F_2,\pardist,p_1,p_2)\}$, $\mathcal{B}_k=\mathcal{B}_{k-1}\cup\{B\}$ and $\mathcal{A}_{k+1}=\mathcal{A}_k\cup \{B\}$.
		\item If $\text{Base}\,{\mathcal{A}_{k+1}}=\bigcup\mathcal{A}_{k+1}$ , stop. Otherwise, repeat with $k=k+1$.
	\end{enumerate}
	
	\begin{figure}
		\centering
		\def\svgwidth{1\textwidth}
		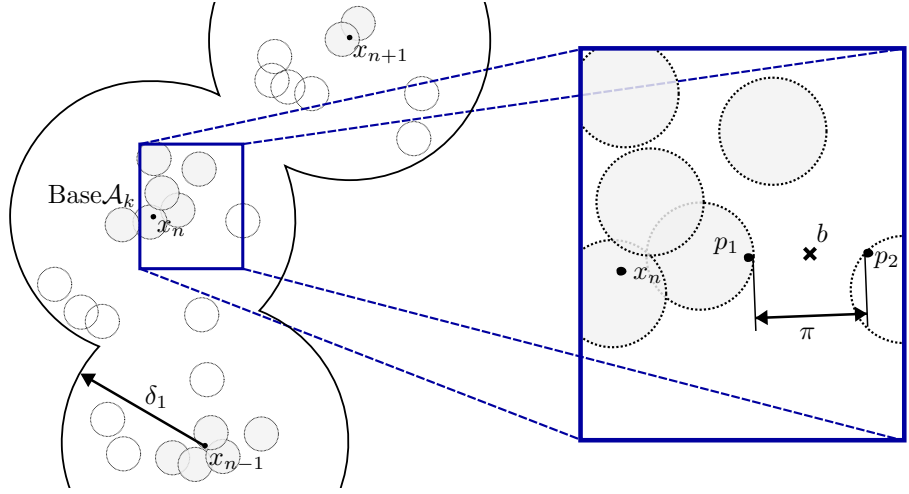
		\caption{Step 2 of the proof of Proposition \ref{lemma:besicovtichpartitions}: Finding the ball in $\mathcal{B}_k$ that corresponds to $\text{Base}\,{\mathcal{A}_k}$. }
	\end{figure}
	
	Since the original cover $\mathcal{A}_1$ is finite, this process will terminate
			after $K \leq \#\text{Cluster}\,{\mathcal{A}} $ steps. Finally, define $\mathcal A$ by removing from $\mathcal{A}_K$ those $A_i$ with $A_i\subset B$ for some $B\in \mathcal{B}_K$. 
	Then $\mathcal{A}\cup\mathcal{B}_K$ forms a covering of $F$ such that each connected component of $\bigcup(\mathcal A \cup \mathcal B_K)$ contains an $x_i$.
	 
	The collections $\mathcal{B}_K$ and $\mathcal{P}_K$ will now be refined slightly to obtain \ref{besi:distinct}.
	Reindex $\mathcal B_K$ (and hence $\mathcal P_K$) such that $\textnormal{diam} B_i \geq \textnormal{diam} B_j$ for all $1\leq i\leq j\leq K$ and define
	\[\mathcal B := \{B_i\in \mathcal B_K: B_i\cap B_j = \emptyset \ \forall 1\leq j<i\}.\]
	Also define
	\[\mathcal{P}:=\{(F_1,F_2,\pardist,p_1,p_2)_k\in\mathcal{P}_K:B_k\in \mathcal B\}. \]
	Then the partitions in $\mathcal{P}$ now satisfy \ref{besi:distinct}, as required. We also have that $\mathcal{B}_K$ is covered by $3\mathcal{B}$ and thus $\mathcal{A}\cup3\mathcal{B}$ forms a cover of $F$.
	
%	Now let $C$ be a connected component of $\mathcal A \cup \mathcal B$ and $x_i \in C$.
%	Observe that
%	\[\mathcal S:=\bigcup\{S\in \mathcal A \cup \mathcal B: S\cap C\neq \emptyset\}\]
%	may be connected by a curve of length
%	\[2\sum \{\textnormal{diam} S: S\in\mathcal S\}\]
%	that begins and end at $x_i$.\unsure{construct tree more explicitly and be very clear how previous lemma is being invoked}
%	Therefore, $\mathcal A \cup 3\mathcal B$ may be connected by a curve of length
%	\[L+\sum_{A\in\mathcal A}2 \textnormal{diam} (A) + \sum_{B\in\mathcal B} 6 \textnormal{diam} (B).\]

	To complete the proof, \ref{besi:lots} is implied by
	\begin{equation}\label{eq:lowerbound}
		\sum_{B_i\in\mathcal{B}}3\textnormal{diam}(B_i)>(1-\varepsilon)\mathcal{H}^1(F).
	\end{equation}  

	Suppose that this inequality is not true. To derive a contradiction, we first must build a particular continuum which we do in three steps. 
	\begin{enumerate}
		\item
		For any $B\in \mathcal{A}$, connect the centre of $B$ to the centres of each ball in $B_\cap$.
		By \ref{martin} there are at most $12C^2$ such balls.
		Repeat this for every ball in $\mathcal{A}$ to obtain a union of continua, $L_1$.
		By \ref{closecover}, $L_1$ satisfies
		\[\mathcal{H}^1(L_1)\leq 12 C^2\sum_{B\in\mathcal A}\textnormal{diam}(B)\leq 96C^3\mathcal{H}^1(F).\]
		Note that, for each connected component of $\cup\mathcal A$, $L_1$ connects the centres of the balls forming that component.
		\item 	For each $B\in\mathcal{A}$, consider the collection 
		\[\text{Nbours}(B):=\left\{\hat{B}\in\mathcal{A}:0<\textnormal{dist}(B,\hat{B})\leq \hat{r}/4\right\}.\]
		These balls are disjoint and are covered by $4B$. Thus, using (\ref{penguin}), we have that
		\[\frac{1}{C}\hat r\cdot\# \text{Nbours}(B)\leq \sum_{\hat B\in \text{Nbours}(B)} \mathcal{H}^1(F\cap \hat{B})\leq \mathcal{H}^1(F\cap 4B)\leq 4C\hat{r} \]
		and hence
		\[\# \text{Nbours}(B)\leq 4C^2.\]
		Thus we may connect the centre of $B$ to the centre of each ball in $\text{Nbours}(B)$ using a continuum of length at most $4C^2\cdot 3 \hat{r}$. Repeat for each $B\in \mathcal{A}$ to create a union $L_2$ of disjoint continua satisfying
		\[\mathcal{H}^1(L_2)\leq 6 C^2\sum_{B\in\mathcal A}\textnormal{diam}(B)\leq 72C^3\mathcal{H}^1(F).\]
		Note that, for each connected component of $(\cup\mathcal A)_{\boldsymbol{\underline{\hat r/4}}}$, $L_1\cup L_2$ connects the centres of the balls forming that component.
		\item For $B\in 3\mathcal{B}$, consider the elements $C_1,\ldots,C_N$ of $\text{Cluster}\,\mathcal{A}$ that intersect $B$. For each $C_n$, pick a ball $B(a_n,\hat r)\in \mathcal{A}$ that is a subset of $A_n$ and that intersects $B$. Note that by construction of $B$ we necessarily have
		\[d(a_n,a_k)\geq \pardist,\qquad \forall n\neq k\]
		and
		\[B(a_n,\pardist/2)\subset B_{\boldsymbol{\underline{\hat r+\pardist/2}}}.\]
		Hence
		\begin{align*}
			\frac{N}{C}\pardist/2\leq \sum_{n=1}^{N}\mathcal{H}^1(F\cap B(a_n,\pardist/2))&\leq \mathcal{H}^1(F\cap B_{\boldsymbol{\underline{\hat r+\pardist/2}}})\\
			&\leq C(3(2R+1)\pardist/2+\hat{r}+\pardist/2)
			.
		\end{align*}
		Since, by contruction, $\pardist\geq \hat{r}/4$, this implies
		\[N\leq C^2(3(2R+1)+2\hat{r}/\pardist+1)\leq C^2(3(2R+1)+9).\]
		Connect the centre of $B$ to each $a_n$ to create a continuum of length at most
		\[C^2(3(2R+1)+9)\cdot (\pardist/2+\hat r)\leq 6C^2(3(2R+1)+9)\pardist.\]
		Repeat this for every $B\in 3\mathcal{B}$ to form a union of continua $L_3$ which satisfies
		\begin{align*}
			\mathcal{H}^1(L_3)&\leq C^2(3(2R+1)+9)\sum_{B\in 3\mathcal B} 3 \textnormal{diam}(B)\\
			&\leq 6C^2(R+2)(1-\varepsilon)\mathcal{H}^1(F).
		\end{align*}
	\end{enumerate} 
	
 All together we have a continuum $\tilde L=L\cup L_1\cup L_2\cup L_3$ such that
	\begin{align*}
		\mathcal{H}^1(\tilde L)&\leq \mathcal{H}^1(L)+96C^3\mathcal{H}^1(F)+72C^3\mathcal{H}^1(F)+6C^2(R+2)(1-\varepsilon)\mathcal{H}^1(F)\\
		&< \mathcal{H}^1(L)+200C^3(R+2)\mathcal{H}^1(F).
	\end{align*}
	Thus, \eqref{chalky} gives
	\begin{equation*}
		\mathcal H^1(F\cap \tilde{L}_{\textbf{\underline{r}}} ) < \frac{\varepsilon}{2} \mathcal H^1(F).
	\end{equation*}
	Since each element of $\mathcal{A}$ has diameter $\hat r<r$,
	\[\mathcal{H}^1(F\cap \bigcup \mathcal{A})\leq\mathcal H^1(F\cap \tilde{L}_{\textbf{\underline{r}}} ) < \frac{\varepsilon}{2} \mathcal H^1(F). \]
	Therefore, since $\mathcal A \cup 3\mathcal B$ covers $F$,
	\begin{equation*}
		\mathcal H^1(F\setminus \bigcup 3\mathcal B) \leq \mathcal H^1(F\cap \bigcup \mathcal A) < \frac{\varepsilon}{2} \mathcal H^1(F)
	\end{equation*}
	and so
	\begin{equation}\label{eq:bladder}
		\mathcal{H}^1(F\cap \bigcup 3\mathcal{B})=\mathcal{H}^1(F)-\mathcal{H}^1(F\backslash \bigcup 3\mathcal{B})>(1-\varepsilon/2)\mathcal{H}^1(F).
	\end{equation}
	However, the diameters of the balls in $3\mathcal{B}$ are less than $\delta_0$. Therefore, combining \eqref{delta0} and the opposite of \eqref{eq:lowerbound} gives
	\begin{equation}
		\mathcal{H}^1(F\cap \bigcup 3\mathcal{B})<\sum_{B_i\in\mathcal{B}}3\textnormal{diam}(B_i)+\frac{\varepsilon}{2}\mathcal{H}^1(F)<(1-\varepsilon/2)\mathcal{H}^1(F),
	\end{equation}
	which contradicts (\ref{eq:bladder}).
	\end{proof}

	\section{Connected Tangents Implies Rectifiable}\label{results}

We now prove the main theorem.
The proof is rather technical however the underlying geometric idea is simple: a set cannot both be approximated by quasi-path connected spaces (which is implied by having connected tangents) and have collections of Besicovitch partitions as in Proposition \ref{lemma:besicovtichpartitions} (which is implied by being purely unrectifiable).
\begin{theorem}\label{ultimate}
	Let $(X,d)$ be a complete metric space and $E\subset X$ be $\mathcal{H}^1$-measurable with $\mathcal{H}^1(E)<\infty$. If for $\mathcal{H}^1$-a.e. $x\in E$, $\Theta_*(E,x)>0$ and
	\[\textnormal{Tan}(X,d,\mathcal{H}^1\llcorner E,x)\subset \mathcal{C}^*,\]
	then $E$ is $1$-rectifiable.
\end{theorem}

\begin{proof}
	We proceed by contradiction and assume $E$ is not $1$-rectifiable. First, by  Lemma \ref{lemma:decomprec/unrec}, we may find a purely $1$-unrectifiable set $E'\subset E$ of positive measure. By  Corollary \ref{cor:densityofrestrictedsets}, we have $\Theta_*(E',x)>0$ for $\mathcal{H}^1$-a.e. $x\in E'$, thus $\mathcal{H}^1\llcorner E'$ is asymptotically doubling and by Lemma \ref{cor:restrict},
	\begin{equation*}
		\textnormal{Tan}(\mathcal{H}^1\llcorner E',x)=\textnormal{Tan}(\mathcal{H}^1\llcorner E,x)\subset \mathcal{C}^*\quad \text{for }\mathcal{H}^1\text{-a.e.}\,x\in E'.
	\end{equation*} 
	Therefore we satisfy the requirements of Lemma \ref{decomposemebaby}.
	Thus, for $\delta=1/6$, we may find a set $F\subset E'$ and numbers $\hat{d},R>0$ such that for $\mathcal{H}^1$-a.e. $x\in F$,
	\begin{align*}
		\ref{me} & \quad \hat{d}\leq\Theta_*(F,x),\, \Theta^*(F,x)\leq 1;\\
		\\
		\ref{you} & \quad \textnormal{Tan}(\mathcal{H}^1\llcorner F,x)\subset \mathcal{Q}_*(1/6,R,1/\hat{d})\cap \mathcal{C}^*.
	\end{align*}
	By the virtue of the fact that $Q_*$ are nested, $\mathcal{Q}_*(\delta,R,C)\subset\mathcal{Q}_*(\delta,R+1,C) $, we can assume that $R$ is such that $8/(R+1)<\hat{d}/12$.
	We may then find a further subset $G\subset F$ with
	\begin{equation}\label{getbigg}
		\mathcal{H}^1(F\backslash G)< \frac{\hat{d}^2(1-\varepsilon)}{72(2R+3)+1}\mathcal{H}^1(F)
	\end{equation}  
	and an $r_0>0$ such that for all $x\in G$ and $r<r_0$
	\begin{equation}\label{eq:upperdensity}
		\hat{d}r\leq\mathcal{H}^1(F\cap B(x,r))\leq 4r,
	\end{equation}
	and such that, by Proposition \ref{prop:tancompact},
	\begin{equation}\label{eq:crumpet}
		d_*(T_{r}(\mathcal{H}^1\llcorner F,d,x),\textnormal{Tan}(\mathcal{H}^1\llcorner F,x))<\frac{1}{R+1}.
	\end{equation}
	
	Fix $\varepsilon>0$ and $\delta\leq\min\{r_0,\hat d\}/2$. Use  Lemma \ref{lemma:besicovtichpartitions} to find a collection $\mathcal{P}$ of Besicovitch partitions for $G$ satisfying
	\begin{equation*}
		\ref{besi:distinct}\quad P_i, P_j \text{ are }(R+1)-\text{disjoint},\hspace{1.5cm} \ref{besi:bounded}\quad\pardist_i<\delta,
	\end{equation*} 
	\begin{equation*}
		\ref{besi:lots}\quad\sum_i \pardist_i>\frac{1-\varepsilon}{3(2R+3)}\mathcal{H}^1(G). 
	\end{equation*}

	For a moment fix a $P=(G_{1},G_{2},\hat \pardist,p_{1},p_{2})$ in $\mathcal{{P}}$ and set $\pardist=2\hat \pardist$. Because $\pardist<r_1$, we can find $(\nu,\rho,y)\in \textnormal{Tan}(\mathcal{H}^1\llcorner F,p_{1})$ such that (\ref{eq:crumpet}) holds for $r=\pardist$. Then by Theorem \ref{theorem:largesubsets}, there exists compact 
	$p_1\in K_F\subset F \cap B(p_1, \pardist(R+1))$ and $y\in K_\nu \subset \supp\nu\cap B(y,R+1)$
	such that 
	\begin{align*}
		\ref{largesubsets1}\quad &\max\left\{\frac{\mathcal{H}^1(F\cap B(p_1,\pardist (R+1))\backslash K_F)}{\mathcal{H}^1(F\cap B(p_1,\pardist))},\nu(B(y,R+1)\backslash K_\nu)\right\}<\frac{1}{R+1}, \\ 
		\\
		\ref{largesubsets2}\quad &d_{pmGH}\left((K_F,\frac{\mathcal{H}^1\llcorner K_F}{\mathcal{H}^1(F\cap B(p_1,\pardist))},p_1),(K_\nu,\nu\llcorner K_\nu,y)\right)<\frac{3}{R+1}.
	\end{align*} 
	
	By \ref{largesubsets2} and Proposition \ref{prop:embeddingsandconvergence}, there exists a metric space $(Z,\lambda)$ and isometric embeddings
	\[(K_F,p_1),(K_\nu,y)\rightarrow(\bar K_F,z),(\bar K_\nu,z)\subset (Z,\lambda,z)\] 
	such that
	\begin{equation}\label{pillow}
		F_z\left(\frac{\mathcal{H}^1\llcorner \bar K_F}{\mathcal{H}^1(F\cap B(p_1,\pardist))},\bar \nu\llcorner \bar K_\nu\right)+d_{H(z)}(\bar K_F, \bar K_\nu)<\frac{3}{R+1},
	\end{equation}
	where for clarity and ease of notation, the image of any set or point under these embeddings is denoted with a bar. 
	
	We now demonstrate the existence of points in both $\bar K_F$ and $\bar K_\mu$ near $\bar{p}_2$.
	
	Firstly, note that by our choice of $R$ we now satisfy the assumptions of Proposition \ref{close?} where $\varepsilon=3/(R+1)$ and  $\delta=1/6$, thus there exists $s\in K_F\cap B_{d/\pardist}(p_2,1/6)$. This point combined with (\ref{pillow}), gives us the existence of $\bar t\in \bar K_\nu$ such that $\lambda(\bar s,\bar t)\leq 3/(R+1)$. Thus, using the triangle inequality we find that
	\begin{align*}
		\lambda(z,\bar t)&\leq \lambda(z,\bar p_2)+\lambda(\bar p_2,\bar s)+\lambda(\bar s,\bar t)\\
		&= d_{d/\pardist}(p_1,p_2)+d_{d/\pardist}(p_2,s)+\lambda(s,t)\\
		&\leq 1/2+1/6+3/(R+1)\leq 1.
	\end{align*} 
	Thus we have that $t\in K_\nu$ is such that $\rho(y,t)\leq 1$. Once again, due to our choice of $R$, we satisfy the assumptions of Corollary \ref{pathypath}, with $\varepsilon=3/(R+1)$ and  $\delta=1/6$, thus we have that $K_\nu\cap B(y,R+1/12)$ contains a $1/3$-quasi-path connecting $y$ and $t$.
	Let $K_{G_1}=K_F\cap G_1$ and $K_{G_2}=K_F\cap G_2$, then we have $\textnormal{dist}(\bar K_{G_1},\,\bar K_{G_2})\geq 1/2.$ Because $\bar K_\nu$ contains a $1/3$-quasi-path between $z$ and $\bar t$, there must be some $\bar q\in \bar K_\nu$ such that
	\[\textnormal{dist}(\bar K_{G_1},\bar q)\geq \frac{1}{6}\quad\text{and}\quad\textnormal{dist}(\bar K_{G_2},\bar q)\geq \frac{1}{6}\]
	and $B(\bar q,1/12)\subset B(z,R+1/24)\subset B(z,R+1)$.
	Thus there exists the required points near $\bar{p}_2$.
	
%	We have that $B(\bar q,3/(R+1))\subset B(z,R+1/12+3/(R+1))\subset B(z,R+1)$, so once again by (\ref{pillow}), we must have some $\bar w\in B(\bar q,3/(R+1))\cap \bar K_F$. 	
%	
%	 
%	By construction of $w$, we have that
%	\unsure{\ref{bikelight} requires $3/(R+1)+M<1/12.$ Try M=1/24.}
%	\begin{equation}\label{bikelight}
%		B(\bar q,M)\subset B(\bar w,1/12).
%	\end{equation}
%	and
%	\begin{equation}
%		\textnormal{dist}_{d/\pardist}(G_1,w)\geq  \frac{1}{6}-\frac{3}{R+1}>\frac{1}{12} \quad \text{and likewise}\quad \textnormal{dist}_{d/\pardist}(G_2,w)>\frac{1}{12}
%	\end{equation}
%	implying that $B_{d/\pardist}(w,1/12)\cap G=\emptyset$. 
	
	Now, since $\nu\in\mathcal{A}(1/\hat d)$,
	\begin{equation}
		\nu(B(q,1/12)\cap K_\nu)\geq \frac{\hat{d}}{12}-\frac{1}{R+1}
	\end{equation}
	which combined with (\ref{pillow}) implies that
	\begin{equation}
		\frac{\mathcal{H}^1\llcorner \bar K_F(B(\bar q,1/12))}{\mathcal{H}^1(F\cap B(p_1,\pardist))}\geq \nu(B(\bar q,1/12)\cap K_\nu)-\frac{3}{R+1}\geq \frac{\hat{d}}{12}-\frac{4}{R+1}
		.
	\end{equation}
	Therefore,
	\begin{align*}
		\mathcal{H}^1(F\cap B_{d/\pardist}(p_1,R/2)\backslash G_1\cup G_2)&\geq (\frac{\hat{d}}{12}-\frac{4}{R+1})\mathcal{H}^1(F\cap B(p_1,\pardist))\\
		&\geq (\frac{\hat{d}}{12}-\frac{4}{R+1})\hat{d}\pardist>\frac{\hat d^2}{24}\pardist.
	\end{align*}

	Since the Besicovitch partitions are disjoint (condition \ref{besi:distinct}), the balls $B_{d/\pardist}(p_1,R/2)$ corresponding to each partition are disjoint.
	  We may repeat the same process for each partition $P_i\in \mathcal{P}$ and use \ref{besi:lots} to obtain
	\begin{align*}
		\mathcal{H}^1(F\backslash G)&\geq \frac{\hat d^2}{24}\cdot \frac{1-\varepsilon}{3(2R+3)}\mathcal{H}^1(G)\\
		&=\frac{\hat{d^2}(1-\varepsilon)}{72(2R+3)}(\mathcal{H}^1(F)-\mathcal{H}^1(F\backslash G)).
	\end{align*}
	This implies
	\begin{equation*}
		\mathcal{H}^1(F\backslash G)\geq \frac{\hat{d}^2(1-\varepsilon)}{72(2R+3)+\hat{d}^2(1-\varepsilon)}\mathcal{H}^1(F)\geq \frac{\hat{d}^2(1-\varepsilon)}{72(2R+3)+1}\mathcal{H}^1(F)
	\end{equation*}
	which contradicts (\ref{getbigg}) and concludes the proof.	
\end{proof}

Finally, we restate the necessary part of \cite[Theorem 1.2]{bate}, with which we obtain our final statement. For $K\geq 1$, let $\textnormal{biLip}(K)^*$ be the set of all pointed metric measure spaces whose metric is $K$-bi-Lipschitz equivalent to the Euclidean norm.
\begin{theorem}[Theorem 1.2 \cite{bate}]\label{chips?}
	
	Let $(X,d)$ be a complete metric space, $n\in \mathbb{N}$ and let $E\subset X$ be $\mathcal{H}^n$-measurable with $\mathcal{H}^n(E)<\infty$. Then $E$ is $n$-rectifiable if and only if  for $\mathcal{H}^n$-a.e. $x\in E$, $\Theta_*^n(E,x)>0$ and there exists a $K_x\geq 1$ such that
	
	\[\textnormal{Tan}(X,d,\mathcal{H}^n\llcorner E,x)\subset \textnormal{biLip}(K_x)^*.\]
	
\end{theorem}
\begin{proof}[Proof of Theorem \ref{lovelytheorem}]
	We have that $\text{biLip}(K)^*$ is a subset of $\mathcal{C}^*$ so simply combine Theorem \ref{chips?} and Lemma \ref{ultimate}.
\end{proof}

%\begin{figure}
%	\centering
%	\def\svgwidth{1\textwidth}
%	\import{}{Ultimatecomic.pdf_tex}
%\end{figure}
%
%\begin{figure}
%	\centering
%	\def\svgwidth{1\textwidth}
%	\import{}{Ultimatecomic2.pdf_tex}
%\end{figure}

	\bibliographystyle{plain}
	\bibliography{biblio}

@book{luigi,
	author    = "L. Ambrosio and P. Tilli",
	title     = "Topics on Analysis in Metric Spaces",
	publisher = "Oxford University Press",
	year      = "2004",
}

@book{book:mattila,
	author    = "Mattila, P.",
	title     = "Geometry of Sets and Measures in Euclidean Spaces",
	publisher = "Cambridge studies in advanced mathematics",
	year      = "1995",
}

@book{book:maggi,
	author    = "Maggi, F.",
	title     = "Sets of Finite Perimeter and Geometric Variational Problems",
	publisher = "Cambridge studies in advanced mathematics",
	year      = "2012",
}

@book{book:heinonen,
 Author = {Heinonen, J and Koskela, P and Shanmugalingam, N and Tyson, J},
 Title = {Sobolev spaces on metric measure spaces. {An} approach based on upper gradients},
 Volume = {27},
 Year = {2015},
 Publisher = {Cambridge University Press},
 Language = {English},
 DOI = {10.1017/CBO9781316135914},
 zbMATH = {6397370},
 Zbl = {1332.46001}
}

@article{bate,
	author = "Bate, D.",
	title = "Characterising rectifiable metric spaces using tangent spaces",
	journal = "Invent. math.",
	volume = "230",
	pages = "995–1070",
	year = "2022",
	DOI = "https://doi.org/10.1007/s00222-022-01136-7"
}

@article{paper:mattila,
	author = "Mattila, P.",
	title = "Hausdorff $m$ regular and rectifiable sets in $n$-space.",
	journal = "Trans. Amer. Math. Soc.",
	volume = "205",
	pages = "9263-274",
	year = "1975",
	DOI = ""
}

@article{paper:marstrand,
	author = "Marstrand, J.M.",
	title = "Hausdorff two-dimensional measure in $3$-space.",
	journal = "Proc. London Math. Soc. (3)",
	volume = "11",
	pages = "91-108",
	year = "1961",
	DOI = ""
}

@article{paper:besi2,
	author = "Besicovitch, A.S.",
	title = "On the fundamental geometrical properties of linearly measurable plane sets of points (II)",
	journal = "Math. Ann.",
	volume = "115(1)",
	pages = "296-329",
	year = "1938",
	DOI = "https://doi.org/10.1007/BF01448943"
}

@article{paper:kirchheim,
	author = "Kirchheim, B.",
	title = "Rectifiable metric spaces: local structure and regularity of the Hausdorff measure.",
	journal = "Proc. Amer. Math. Soc.",
	volume = "121(1)",
	pages = "113-123",
	year = "1994",
	DOI = "https://doi.org/10.2307/2160371"
}

@article{paper:preiss,
	author = "Preiss, D.",
	Title = "Geometry of measures in {{\(R^ n:\)}} {Distribution}, rectifiability, and densities",
	journal = "Ann. of Math. (2)",
	volume = "125(3)",
	pages = "537-643",
	year = "1987",
	DOI = "https://doi.org/10.2307/1971410"
}

@Article{zbMATH00120128,
 Author = {Preiss, David and Ti{\v{s}}er, Jaroslav},
 Title = {On {Besicovitch}'s {{\(\frac{1}{2}\)}}-problem},
 FJournal = {Journal of the London Mathematical Society. Second Series},
 Journal = {J. Lond. Math. Soc., II. Ser.},
 ISSN = {0024-6107},
 Volume = {45},
 Number = {2},
 Pages = {279--287},
 Year = {1992},
 Language = {English},
 DOI = {10.1112/jlms/s2-45.2.279},
 zbMATH = {120128},
 Zbl = {0762.28003}
}

@misc{2404.17536,
Author = {Camillo De Lellis and Federico Glaudo and Annalisa Massaccesi and Davide Vittone},
Title = {Besicovitch's 1/2 problem and linear programming},
Year = {2024},
Eprint = {arXiv:2404.17536},
}

@article{azzam,
   title={A characterization of 1-rectifiable doubling measures with connected supports},
   volume={9},
   ISSN={2157-5045},
   url={http://dx.doi.org/10.2140/apde.2016.9.99},
   DOI={10.2140/apde.2016.9.99},
   number={1},
   journal={Analysis \&; PDE},
   publisher={Mathematical Sciences Publishers},
   author={Azzam, Jonas and Mourgoglou, Mihalis},
   year={2016},
   month=feb, pages={99–109} }

	\end{document}